\documentclass[a4paper, 10pt]{amsart}
\usepackage[]{amsmath}
\usepackage{amssymb}
\usepackage{enumerate}
\usepackage{tabularx}
\usepackage[]{color}
\usepackage[left=2.5cm, right=2.5cm, top=2.5cm, bottom=2.5cm]{geometry}
\usepackage[colorlinks]{hyperref}
\usepackage{tikz}
\usepackage{multirow}
\usepackage{diagbox}
\usepackage{subcaption}
\usepackage[ruled, vlined]{algorithm2e}
\usepackage{xcolor}

\allowdisplaybreaks[4]

\newcommand\substitute[2]{{#2}}
\newcommand\justsubstitute[2]{#2}
\newcommand\revise[2]{{#2}}

\title[Preconditioned RDA Method]{Preconditioned
NonSymmetric/Symmetric Discontinuous Galerkin Method for Elliptic
Problem with Reconstructed Discontinuous Approximation}

\author[R. Li]{Ruo Li} \address{CAPT, LMAM and School of Mathematical
Sciences, Peking University, Beijing 100871, P.R. China;
Chongqing Research Institute of Big Data, Peking University, Chongqing
401121, P.R. China}
\email{rli@math.pku.edu.cn}

\author[Q.-C. Liu]{Qicheng Liu} \address{School of Mathematical
Sciences, Peking University, Beijing 100871, P.R. China}
\email{qcliu@pku.edu.cn}

\author[F.-Y. Yang]{Fanyi Yang} \address{College of Mathematics
, Sichuan University, Chengdu 610065, P.R. China}
\email{yangfanyi@scu.edu.cn}

\newcommand{\bm}[1]{\boldsymbol{#1}}
\newcommand{\bmr}[1]{\bm{\mr{#1}}}
\newcommand{\lj}{[ \hspace{-2pt} [}
\newcommand{\rj}{] \hspace{-2pt} ]}
\newcommand{\mb}[1]{\mathbb{#1}}
\newcommand{\mc}[1]{\mathcal{#1}}
\newcommand{\mr}[1]{\mathrm{#1}}
\newcommand{\jump}[1]{\lj #1 \rj}
\newcommand{\aver}[1]{ \{#1\}  }
\newcommand{\wt}[1]{ \widetilde{ #1}}
\newcommand{\wh}[1]{ \widehat{ #1}}

\newcommand{\DGenorm}[1]{ \| #1\|_{\mr{DG}}}
\newcommand{\DGtenorm}[1]{|\!|\!| #1 |\!|\!|_{\mr{DG}}}

\renewcommand{\d}[1]{\mathrm d \boldsymbol{#1}}

\def\dim{\ifmmode \mathrm{dim} \else \text{dim}\fi}
\def\rank{\ifmmode \mathrm{rank} \else \text{rank}\fi}
\def\MTh{\mc{T}_h}

\def\MEh{\mc{E}_h}

\def\MEhB{\mc{E}_h^{\backslash \Gamma}}
\def\un{\bm{\mr n}}

\def\MEhI{\mc{E}_h^I}
\def\MEhB{\mc{E}_h^B}

\def\btau{\bm{\tau}}

\def\lLambda{\overline{\Lambda}}

\newtheorem{assumption}{Assumption}
\newtheorem{theorem}{Theorem}
\newtheorem{lemma}{Lemma}

\newtheorem{remark}{Remark}

\definecolor{orange}{rgb}{1, 0.5, 0}

\begin{document}

\maketitle

\begin{abstract}
  In this paper, we propose and analyze an efficient preconditioning
  method for the elliptic problem based on the reconstructed discontinuous
  approximation method. 
  This method is originally proposed in [Li et al., J. Sci. Comput. 80(1), 2019] 
  that an arbitrarily high-order approximation space with one unknown
  per element is reconstructed by solving a local least squares
  fitting problem.
  This space can be directly used with
  the symmetric/nonsymmetric interior penalty discontinuous Galerkin
  methods. The least squares problem is modified in this paper, which
  allows us to establish a norm equivalence result between the
  reconstructed high-order space and the piecewise constant space.
  This property further inspires us to construct a preconditioner from
  the piecewise constant space. The preconditioner is shown to be
  optimal that the upper bound of
  the condition number to the preconditioned symmetric/nonsymmetric
  system is independent of the mesh size. In addition, 
  we can enjoy the
  advantage on the efficiency of the approximation in number of
  degrees of freedom compared with the standard DG method.
  Numerical experiments are provided to demonstrate the validity of
  the theory and the efficiency of the proposed method.

\noindent \textbf{keywords}: discontinuous Galerkin method;
reconstructed discontinuous approximation; symmetric/ \\ nonsymmetric
interior penalty method; preconditioning; least squares fitting;

\end{abstract}


\section{Introduction}
\label{sec_introduction}
In recent years, there have been extensive studies focused on the
development of discontinuous Galerkin (DG) methods, and now DG methods
are the very standard numerical methods in solving a variety of
partial differential equations; \substitute{}{we refer to
\cite{Cangiani2017hpversion,Cockburn2000development,Riviere2008discontinuous} for some monographs.} 
The DG methods use totally discontinuous piecewise polynomial
approximation space, which gives several advantages over other
types of finite element methods \cite{Antonietti2017uniform}.  
The full discontinuity across the element faces brings the advantage 
of the flexibility on the mesh partition. The DG methods can be easily
applied to the polygonal mesh and the mesh with hanging nodes, see
\cite{Cangiani2014hp,Cangiani2017hpversion,Lipnikov:2013} for some examples.
Such meshes ease the triangulation of complex geometries and curved
boundaries. The implementation of the DG space is easy because the
basis functions are entirely local. 
On the other hand, the DG method may be computationally expensive and
inefficient on the approximation because of the significant increase
in the number of degrees of freedom \cite{Hughes2000comparison}.
This shortcoming is remarkable
especially for the high-order scheme. In addition, the final linear
system is ill-conditioned with the condition number grows like
$O(h^{-2})$ \cite{Castillo2002performance}.

To overcome the issue on the huge number of degrees of freedom, we
proposed a reconstructed discontinuous approximation method, where a 
high-order approximation space is constructed with only one unknown
per element \cite{Li2016discontinuous}. 
The main step of the reconstruction is solving a local least squares
fitting problem on the element patch. This space is a small subspace
of the standard discontinuous piecewise polynomial space, and inherits
the flexibility on the mesh partition. We prove the optimal
approximation estimate, while the number of degrees of freedom is
substantially reduced to the number of elements in the mesh. 
This approach has been successfully applied to a series of classical
problems \cite{Li2017discontinuous,Li2019eigenvalue,Li2020interface,Li2023curl,Li2019least}. 
In this paper, we modify the local least
squares problem in \cite{Li2016discontinuous} by 
adding a constraint at the barycenter. This modification 
essentially brings us the
non-degeneracy to the reconstruction operator, which allows us to
further develop an efficient preconditioner based on the piecewise
constant space. In solving the elliptic problem with the reconstructed
space, the first conclusion from this property is the $O(h^{-2})$
condition number estimate, which is the same as the standard DG
method.

Solving the resulting linear system efficiently is of special concern
in finite element methods, and the most common
method is to construct a proper preconditioner that can significantly
reduce the condition number in the preconditioned system.
The Schwarz method is one of the popular methods for
preconditioning the linear system arising from the DG discretization.  
In \cite{Feng2001additive,Antonietti2007schwarz,Antonietti2011domain,Antonietti2014domain,Barker2011overlapping,Karakashian2017additive}, a wide class of overlapping and
non-overlapping domain decomposition methods are proposed and
analyzed. The condition number of the preconditioned linear system has
the bounds $O(H/\delta)$ and $O(H/h)$, where $H, h$ and
$\delta$ stand for the granularity of the coarse and fine grids and
the size of the overlap, respectively. In addition, there are other kinds
of preconditioners based on the additive Schwarz framework designed
for DG methods, like balancing domain decomposition methods,
balancing domain decomposition with constraints and the auxiliary
space methods, see e.g. \cite{Dryja2008BDD,Dryja2007BDDC,Dobrev2006two,Brix2008multilevel,Antonietti2017uniform}.  Multigrid
methods are also the wildly used approaches to design preconditioners
for the DG discretization. The main idea behind these methods is to
correct the error after a few smoothing iterations on a coarser grid
\cite{Xu2017algebraic}. We refer to
\cite{Gopalakrishnan2003multilevel,Antonietti2017multigrid,Brenner2011multigrid} for geometrical multigrid methods and
\cite{Olson2011smoothed,Antonietti2020algebraic} for algebraic
multigrid methods. 

In this paper, we construct and analyze an efficient preconditioning
method for the reconstructed space in solving the elliptic system.
We first employ the symmetric/nonsymmetric interior penalty DG method
in the numerical scheme. The error estimates under error measurements 
are standard under the Lax-Milgram framework.  The main feature of 
the proposed scheme is that only one degree of freedom is involved per 
element, which first gives a higher efficiency on the finite element
approximation in number of degrees of freedom. We conduct a series of
numerical tests to show that our method can use much fewer degrees of
freedom to achieve the comparable numerical error.
Second, the size of the reconstructed space is always the same as the
piecewise constant space. From this fact, 
we can construct a preconditioner from the piecewise constant space to
any high-order reconstructed space. 
The low-order preconditioning is also a classical
technique in finite element methods for preconditioning the high-order
discretization, see \cite{Chalmers2018low,Pazner2023low,Pazner2020efficient}.
In our method,
the modified local least squares problem enables us to prove a norm
equivalence between the piecewise constant space and the high-order
reconstructed space on the same mesh.
This crucial property further allows us to 
prove that the preconditioner is optimal for any high-order accuracy
in the sense that the upper bound of the condition number to the 
preconditioned system is independent of the mesh size. 
A series of numerical experiments including an example on the
polygonal mesh are presented to demonstrate the efficiency of the
preconditioner in two and three dimensions.

The rest of this paper is organized as follows. In Section
\ref{sec_preliminaries}, we introduce the notation and recall some
inequalities. Section \ref{sec_space} introduces the reconstructed
discontinuous approximation method including the construction to the
reconstruction operator and some basic properties of the space.
Section \ref{sec_elliptic} presents the symmetric/nonsymmetric
interior penalty DG methods with the reconstructed space to the
elliptic problem. The preconditioning method is also given in this
section.  In Section \ref{sec_numericalresults}, the accuracy and the
efficiency of the proposed method are illustrated by a series of
numerical tests.  A brief conclusion is given in Section
\ref{sec_conclusion}.


\section{Preliminaries}
\label{sec_preliminaries}
Let $\Omega \subset \mb{R}^d(d = 2, 3)$ be a bounded convex polygonal
(polyhedral) domain with the boundary $\partial \Omega$.  We denote by
$\MTh$ a quasi-uniform triangulation over the domain $\Omega$. Let
$\MEhI$ and $\MEhB$ be the collections of all $d-1$ dimensional
interior faces in the partition $\MTh$, and faces lying on the
boundary $\partial \Omega$, respectively. 
We further set $\MEh := \MEhI \cup \MEhB$ as the set of all faces. 
For any element $K \in \MTh$ and any face $e \in \MEh$, we set $h_K :=
\text{diam}(K)$ and $h_e := \text{diam}(e)$ as their diameters, and we
let $\rho_K$ be the radius of the largest inscribed disk (ball) in
$K$.
We denote by $h:= \max_{K \in \MTh} h_K$ the mesh size to $\MTh$,
and by $\rho := \min_{K \in \MTh} \rho_K$. 
The mesh $\MTh$ is
assumed to be quasi-uniform in the sense that there exists a constant
$C_\sigma > 0$ such that $h \leq C_\sigma \rho$.

The quasi-uniformity of $\MTh$ brings us the following fundamental
estimates: the trace estimate and the inverse estimate: there exist
constants $C$ independent of $h$ such that 
\begin{align}
  \| v \|_{L^2(\partial K)}^2 &\leq C ( h_K^{-1} \| v \|_{L^2(K)}^2 +
  h_K \| v \|_{H^1(K)}^2), \quad \forall v \in H^1(K), \quad \forall K
  \in \MTh,
  \label{eq_trace} \\
  \|  v \|_{H^q(K)} &\leq C h_K^{p - q} \| v \|_{H^p(K)}, \quad q \geq
  p \geq 0, \quad \forall v \in \mb{P}_m(K), \quad \forall K \in \MTh,
  \label{eq_inverse}
\end{align}
where $\mb{P}_m(\cdot)$ denotes the space of polynomials of degree
less than $m$.

Next, we introduce the following trace operators associated with the
weak forms, which are commonly used in the DG framework.
Let $e \in \MEhI$ be any interior face shared by 
two elements $K^+$ and $K^-$ with the unit outward normal vectors
$\un^+$ and $\un^-$ along $e$, respectively. 
For any piecewise smooth scalar-valued function $v$ and vector-valued
function $\btau$, 
the jump operator $\jump{\cdot}$ and the average operator 
$\aver{\cdot}$ are defined as
\begin{displaymath}
  \begin{aligned}
    \jump{v}|_e &:= v^+|_{e} \un^+ + v^-|_e \un^-, \quad \aver{v}|_e
    := \frac{1}{2}(v^+|_e + v^-|_e), \\
    \jump{\bm{\tau}}|_e &:= \bm{\tau}^+|_e \cdot \un^+ +
    \bm{\tau}^-|_e \cdot \un^-, \quad \aver{\bm{\tau}}|_e :=
    \frac{1}{2}(\bm{\tau}^+|_e + \bm{\tau}^-|_e ),
  \end{aligned} \quad \forall e \in \MEhI,  
\end{displaymath}
where $v^{\pm} := v|_{K^{\pm}}$, $\btau^{\pm} :=
\btau|_{K^{\pm}}$.
For any boundary face $e \in \MEhB$, the trace operators are modified
as 
\begin{displaymath}
  \begin{aligned}
    \jump{v}|_e := v|_e \un, \quad \aver{v}|_e := v|_e , \quad
    \jump{\bm{\tau}}|_e := \bm{\tau}|_e \cdot \un, \quad
    \aver{\bm{\tau}}|_e  := \bm{\tau}|_e, \quad \forall e \in \MEhB,
  \end{aligned}
\end{displaymath}
where $\un$ is the unit outward normal to $e$.

For a bounded domain $D$, $L^2(D)$ and $H^r(D)$ denote the usual
Sobolev spaces with the exponent $r \geq 0$, and we also follow their
associated inner products, seminorms and norms. Throughout this
paper, $C$ and $C$ with subscripts are denoted to be generic constants
that may vary in different occurrences but are always independent of
the mesh size $h$. 
We also point out that these constants may depend on
the polynomial degree, the quasi-uniformity measure $C_\sigma$ and the
coefficient in the problem \eqref{eq_elliptic}, such as the constants
appearing in the trace and inverse estimates \eqref{eq_trace} - 
\eqref{eq_inverse}. We refer to \cite{Cangiani2017hpversion} for more
details about the constants in \eqref{eq_trace} - \eqref{eq_inverse}.

In this paper, we are concerned with the elliptic problem defined on
$\Omega$, which seeks $u$ such that
\begin{equation}
  \begin{aligned}
    -\nabla \cdot (A \nabla u) &= f, && \text{in } \Omega, \\
    u &= g, && \text{on } \partial \Omega, \\
  \end{aligned}
  \label{eq_elliptic}
\end{equation}
where $f \in L^2(\Omega)$ is a given source function and $g \in
H^{3/2}(\partial \Omega)$ is the boundary condition. The coefficient
$A \in \mb{R}^{d \times d}$ is assumed to be a symmetric positive
definite matrix.  By the elliptic regularity theory, the problem
\eqref{eq_elliptic} admits a unique solution in $H^2(\Omega)$.
We are aiming to give a preconditioned interior penalty discontinuous
Galerkin method for the problem \eqref{eq_elliptic}, based on a
reconstructed approximation space.


\section{The Reconstructed Discontinuous Approximation Method}
\label{sec_space}
In this section, we will introduce a linear reconstruction operator to
obtain a discontinuous approximation space for the given mesh $\MTh$.
The reconstructed space can achieve a high-order accuracy while the
number of degrees of freedom always remain the same as the number of
elements in $\MTh$.  The construction of the operator includes two
steps. 

Step 1. For each $K \in \MTh$, we construct an element patch $S(K)$,
which consists of $K$ itself and some surrounding elements. 
The size of the patch is controlled with a given threshold $\# S$. 
The construction of the element patch $S(K)$ is conducted by a
recursive algorithm. We begin by setting $S_0(K) = \{ K \}$, and
define $S_t(K)$ recursively: 
\begin{equation}
  S_t(K) = \bigcup_{K' \in S_{t-1}(K)} \bigcup_{ K'' \in \Delta(K')}
  K'', \quad t = 0,1, \dots
  \label{eq_patch}
\end{equation}
where $\Delta(K):= \{ K' \in \MTh \ | \  \overline{K} \cap \overline{K'}
\neq \varnothing \}$. The recursion stops once $t$ meets the condition 
that the cardinality $\# S_t(K) \geq \# S$, and we let the patch 
$S(K) := S_t(K)$. Applying the recursive algorithm \eqref{eq_patch} 
to all elements in $\MTh$, each $K \in \MTh$ has a large element patch
$S(K)$.

Step 2. For each $K \in \MTh$, we solve a local least squares fitting
problem on the patch $S(K)$. We let $\bm{x}_K$ be the barycenter of
the element $K$ and mark barycenters of all elements as
collocation points.  
Let $I(K)$ be the set of collocation points located inside the domain
of $S(K)$,
\begin{displaymath}
  I(K) := \{ \bm{x}_{K'} \ |\ K' \in S(K) \},
\end{displaymath}
Let $U_h^0$ be the piecewise constant space with respect to $\MTh$,
\begin{displaymath}
  U_h^0 := \{ v_h \in L^2(\Omega) \ | \  v_h|_K \in \mb{P}_0(K), \  
  \forall K \in \MTh\}.
\end{displaymath}
Given a piecewise constant function $v_h \in U_h^0$, for each $K \in
\MTh$, we seek a polynomial of degree $m(m \geq 1)$ by the following
constrained local least squares problem on the patch $S(K)$, 
\begin{equation}
  \begin{aligned}
    \mathop{\arg\min}_{p \in \mb{P}_m(S(K))} &\sum_{\bm{x} \in I(K)} (
    p(\bm{x}) - v_h(\bm{x}))^2, \\
    \text{s.t. } &p(\bm{x}_K) = v_h(\bm{x}_{K}).
  \end{aligned}
  \label{eq_leastsquares}
\end{equation}
We note that the unisolvence of the problem \eqref{eq_leastsquares}
depends on the distribution of collocation points in $I(K)$.  We make
the following assumption of $S(K)$ and $I(K)$. 
\begin{assumption}
  For any element $K \in \MTh$ and any polynomial $p \in
  \mb{P}_m(S(K))$, $p|_{I(K)} = 0$ implies $p|_{S(K)} = 0$.
  \label{as_1}
\end{assumption}
Assumption \ref{as_1} indicates that the points in $I(K)$ deviate from
being located on an algebraic curve (surface) of degree $m$. In
addition, Assumption \ref{as_1} implies the condition $\#I(K) =
\#S(K) \geq \dim(\mb{P}_m(\cdot))$.
The existence and the uniqueness of the solution to
\eqref{eq_leastsquares} are stated in the following lemma.
\begin{lemma}
  For each $K \in \MTh$, the problem \eqref{eq_leastsquares} admits a
  unique solution. 
  \label{le_unique_sol}
\end{lemma}
\begin{proof}
  We mainly prove the uniqueness of the solution since the existence is
  trivial. Let $p_1$ and $p_2$ be the solutions to
  \eqref{eq_leastsquares}. Let $q \in \mb{P}_m(S(K))$ be any
  polynomial such that $q(\bm{x}_K) = 0$, and let $t > 0$ be
  arbitrary. For $i = 1,2$, $p_i + tq$ satisfies the constraint in
  \eqref{eq_leastsquares}. 
  Bringing $p_i + tq$ into \eqref{eq_leastsquares} yields that
  \begin{displaymath}
    \sum_{\bm{x} \in I(K)} (p_i(\bm{x}) + tq(\bm{x})
    - v_h(\bm{x}))^2 \geq \sum_{\bm{x} \in I(K)} (
    p_i(\bm{x}) - v_h(\bm{x}) )^2,
  \end{displaymath}
  which further gives 
  \begin{displaymath}
    \sum_{\bm{x} \in I(K)} -2t q(\bm{x})
    (p_i(\bm{x}) - v_h(\bm{x})) +
    t^2(q(\bm{x}))^2 \geq 0.
  \end{displaymath}
  Because $q$ and $t$ are arbitrary, the above inequality gives us
  that
  \begin{equation}
    \sum_{\bm{x} \in I(K)} q(\bm{x}) (p_i(\bm{x}) -
    v_h(\bm{x})) = 0 \quad \text{and} \quad  \sum_{\bm{x} \in
    I(K)} q(\bm{x}) (p_1(\bm{x}) -
    p_2(\bm{x})) = 0.
    \label{eq_qpiuorth}
  \end{equation}
  Note that $(p_1 - p_2) \in \mb{P}_m(S(K))$ with $(p_1 -
  p_2)(\bm{x}_K) = 0$. Letting $q = p_1 - p_2$ immediately shows that
  $p_1 - p_2$ vanishes at all points in $I(K)$. By Assumption
  \ref{as_1}, we know that $p_1 - p_2 = 0$. This completes the proof.
\end{proof}
We denote by $v_{h, S(K)} \in \mb{P}_m(S(K))$ the solution to
\eqref{eq_leastsquares}.  From the proof of Lemma \ref{le_unique_sol},
it can be seen that $v_{h, S(K)}$ linearly depends on the given
function $v_h$, which allows us to define a local linear operator
$\mc{R}_K^m: U_h^0 \rightarrow \mb{P}_m(S(K))$ such that $\mc{R}_K^m
v_h$ is the solution $v_{h, S(K)}$ of the problem
\eqref{eq_leastsquares} for $\forall v_h \in U_h^0$.

Further, we define a global reconstruction operator $\mc{R}^m$ in an
element-wise manner, which reads
\begin{equation}
  \begin{aligned}
    \mc{R}^m : \,\, U_h^0 &\longrightarrow U_h^m, \\
    v_h &\longrightarrow \mc{R}^m v_h,
  \end{aligned}
  \quad (\mc{R}^m v_h)|_K := (\mc{R}^m_K v_h)|_K, \quad \forall K \in
  \MTh.
  \label{eq_Rm}
\end{equation}
For any element $K \in \MTh$, $(\mc{R}^m v_h)|_K$ is the restriction
of $\mc{R}^m_K v_h$ on $K$. Thus, $(\mc{R}^m v_h)|_K$ has the same
expression as $\mc{R}^m_K v_h$.  By $\mc{R}^m$, any piecewise constant
function $v_h$ is mapped into a piecewise $m$-th degree polynomial
function $\mc{R}^m v_h$. Here $U_h^m := \mc{R}^m U_h^0$ is defined as
the image space, which is actually the approximation space in the
numerical scheme. From the linearity of $\mc{R}^m$, we know that
$\rank(\mc{R}^m) \leq \dim(U_h^0)$. In fact, we can prove that
$\mc{R}^m$ is non-degenerate. 
\substitute{}{
\begin{lemma}
  The operator $\mc{R}^m$ is full-rank, i.e. $\dim(U_h^m) =
  \dim(U_h^0)$.
  \label{le_fullrank}
\end{lemma}
\begin{proof}
  Since $\mc{R}^m$ is linear, it is equivalent to show that if some 
  $v_h \in U_h^0$ satisfy that $\mc{R}^m v_h = 0$, then $v_h = 0$.
  For such $v_h$, the constraint in \eqref{eq_leastsquares} indicates
  that $v_h(\bm{x}_K) = (\mc{R}^m v_h)(\bm{x}_K) = 0$ for every
  element $K \in \MTh$. Then, we conclude that $v_h$ must be the zero
  function, which completes the proof.
\end{proof}
}
We note that Lemma \ref{le_fullrank} is essentially established 
on the constraint in \eqref{eq_leastsquares}, which is the major
improvement compared with our previous methods in
\cite{Li2012efficient,Li2016discontinuous}. This
constraint and the non-degenerate property are also fundamental for us
to develop an \substitute{efficiency}{efficient} preconditioning
method based on the space $U_h^0$.

Let us give a group of basis functions to the reconstructed space
$U_h^m$. 
We denote by $n_e$ the number of elements in $\MTh$, and
there holds $\dim(U_h^m) = n_e$. 
For an element $K \in \MTh$ numbered by the index $i$, we associate
$K$ with a piecewise constant function $e_i \in U_h^0$ such that 
\begin{equation}
  e_i(\bm{x}) = \left\{ \begin{aligned}
    &1, \quad \bm{x} \in K, \\
    &0, \quad \text{otherwise.}
  \end{aligned}
  \right.
  \label{eq_Kei}
\end{equation}
Let $\lambda_j := \mc{R}^m e_j(1 \leq j \leq n_e)$.  
Then, we claim that $\{\lambda_j\}_{j = 1}^{n_e}$ are linearly
independent. Again from the
constraint in \eqref{eq_leastsquares}, we know that
$\lambda_j(\bm{x}_K) = 1$ with $K$ indexed by $j$, and $\lambda_j$
vanishes at other
collocation points. Consider the group of coefficients $\{a_j \}_{j =
1}^{n_e}$ such that
\begin{equation}
  a_1 \lambda_1(\bm{x}) + a_2 \lambda_2(\bm{x}) + \ldots + a_{n_e}
  \lambda_{n_e}(\bm{x}) = 0.
  \label{eq_ajlj}
\end{equation}
We let $\bm{x} = \bm{x}_K$ for each $K \in \MTh$ in \eqref{eq_ajlj},
and we can know that
all $a_j = 0$. The linear independence of $\{ \lambda_j
\}_{j=1}^{n_e}$ is reached. Since $\lambda_j \in U_h^m$ for $1 \leq j
\leq n_e$, we conclude that
$U_h^m = \text{span}(\{ \lambda_j\}_{j = 1}^{n_e})$. Equivalently, $\{
\lambda_j \}_{j=1}^{n_e}$ are basis functions of $U_h^m$.
In \eqref{eq_Kei}, for any element $K'$ such that $K \notin S(K')$,
there holds $e_i|_{S(K')} = 0$, and we can know that $\mc{R}_{K'}^m e_i =
0$. This fact implies that $\lambda_i$ has a compact support that
$\text{supp}(\lambda_i) = \bigcup_{K'|  K \in S(K')} \overline{K'}$.

Next, we extend the operator $\mc{R}^m$ to $C^0$ smooth functions in a
natural way.
Given any $v \in C^0(\Omega)$, we define an associated piecewise
constant function $\wt{v}_h \in U_h^0$ such
that $\wt{v}_h = v$ at all collocation points, i.e. 
$\wt{v}_h(\bm{x}_K) = v(\bm{x}_K)(\forall K \in \MTh)$. Then  
$\mc{R}^m v$ is defined as $\mc{R}^m v:= \mc{R}^m \wt{v}_h$.  By the
basis functions $\{ \lambda_j \}_{j = 1}^{n_e}$, $\mc{R}^m v$ has the
expansion that
\begin{equation}
  \mc{R}^m v = v(\bm{x}_{K_1}) \lambda_1 + v(\bm{x}_{K_2}) \lambda_2 +
  \ldots + v(\bm{x}_{K_{n_e}}) \lambda_{n_e}, \quad \forall v \in
  C^0(\Omega),
  \label{eq_Rmvexpansion}
\end{equation}
where $K_j$ is the element indexed by $j$.

\begin{remark}
  The least squares problem \eqref{eq_leastsquares} does not rely on
  the geometrical shape of the element. The reconstruction process
  can be applied to the polygonal mesh. We refer to
  \cite{DaVeiga2014} for the shape regularity conditions, which bring
  the trace estimate \eqref{eq_trace} and the inverse estimate
  \eqref{eq_inverse} for the polygonal mesh. All estimates in this
  paper can be extended on the shape-regular polygonal mesh without
  any difficulty.
  \label{re_polygonalmesh}
\end{remark}

We present the following stability property of the reconstruction
operator.
\begin{lemma}
  For any element $K \in \MTh$, there holds
  \begin{equation}
    \| \mc{R}^m_K g \|_{L^\infty(S(K))} \leq (1 + 2 \Lambda(m, K)
    \sqrt{ \# S(K)}) \max_{\bm{x} \in I(K)} |g|, \quad \forall g \in
    C^0(\Omega), 
    \label{eq_recons_stability}
  \end{equation}
  where
  \begin{equation}
    \Lambda(m,K) := \max_{p \in \mb{P}_m(S(K))} \frac{\max_{\bm{x} \in
    S(K)} | p(\bm{x}) |}{\max_{\bm{x} \in I(K)} | p(\bm{x})|}.
    \label{eq_Lambda}
  \end{equation}
  \label{le_recons_stability}
\end{lemma}
\begin{proof}
  Let $p = \mc{R}_K^m g \in \mb{P}_m(S(K))$ be the solution to
  \eqref{eq_leastsquares} on $S(K)$. By \eqref{eq_qpiuorth}, there
  holds $\sum_{\bm{x} \in I(K)} q(\bm{x})(p(\bm{x}) -
  g(\bm{x})) = 0$ for any $q \in \mb{P}_m(S(K))$ with 
  $q(\bm{x}_K) = 0$. Setting $q(\bm{x}) = p(\bm{x}) - g(\bm{x}_K)$, we
  know that $\sum_{\bm{x} \in I(K)} (p(\bm{x}) -
  g(\bm{x}_K))(p(\bm{x}) - g(\bm{x})) = 0$. From this equality, we
  find that
  \begin{align*}
    \sum_{\bm{x} \in I(K)} (p(\bm{x}) - g(\bm{x}_K))^2 = \sum_{\bm{x}
    \in I(K)}(g - g(\bm{x}_K))(p - g(\bm{x}_K)).
  \end{align*}
  Applying the Cauchy-Schwarz inequality, we have that
  \begin{equation}
    \sum_{\bm{x} \in I(K)} (p - g(\bm{x}_K) )^2 \leq \sum_{\bm{x} \in
    I(K)} (g - g(\bm{x}_K))^2.
    \label{eq_recons_stability_proof}
  \end{equation}
  Combining \eqref{eq_recons_stability_proof} and the definition
  \eqref{eq_Lambda}, we obtain that
  \begin{displaymath}
    \begin{aligned}
      \| p - g(\bm{x}_K) \|^2_{L^\infty(S(K))} 
      &\leq \Lambda(m, K)^2 \max_{\bm{x} \in I(K)}(p - g(\bm{x}_K))^2
      & \leq \Lambda(m, K)^2 \sum_{\bm{x} \in I(K)} (g - g(\bm{x}_K))^2  \\
      &\leq 4 \Lambda(m, K)^2 \# S(K) \max_{\bm{x} \in I(K)} g^2.
    \end{aligned}
  \end{displaymath}
  Hence, we conclude that
  \begin{displaymath}
    \begin{aligned}
      \| p \|_{L^\infty(S(K))} &\leq \| p - g(\bm{x}_K)
      \|_{L^\infty(S(K))} + | g(\bm{x}_K) | \leq (1 +2  \Lambda(m,K)
      \sqrt{ \# S(K)}) \max_{\bm{x} \in I(K)} |g|,
    \end{aligned}
  \end{displaymath}
  which completes the proof.
\end{proof}
We introduce the constants ${\Lambda}_m$ and $\overline{\Lambda}_m$ by
\begin{equation}
  \Lambda_m := \max_{K \in \MTh} \Lambda(m, K), \quad
  \overline{\Lambda}_m :=  \max_{K \in \MTh} (1 +  \Lambda_m \sqrt{\#
  S(K)}).
  \label{eq_recon_const}
\end{equation}
From the stability estimate \eqref{eq_recons_stability}, we have that 
\begin{equation}
  \|g - \mc{R}_K^m g \|_{L^{\infty}(S(K))} \leq 2\overline{\Lambda}_m
  \inf_{p \in \mb{P}_m(S(K))} \| g - p \|_{L^{\infty}(S(K))}, \quad
  \forall g \in C^0(\Omega).
  \label{eq_bestapp}
\end{equation}
The approximation property of $\mc{R}_K^m$ can be established on
\eqref{eq_bestapp}.
\begin{lemma}
  There exists a constant $C$ such that 
  \begin{equation}
    \|g - \mc{R}^m_K g \|_{H^q(K)} \leq C \overline{\Lambda}_m
    h_K^{m+1-q} \| g \|_{H^{m+1}(S(K))}, \quad 0 \leq q \leq m, \quad
    \forall K \in \MTh.
    \label{eq_Rmapp}
  \end{equation}
  \label{le_Rmapp}
\end{lemma}
The proof is quite formal and we refer to \cite[Theorem
3.3]{Li2012efficient} for details.

By \eqref{eq_bestapp} - \eqref{eq_Rmapp}, the polynomial
$\mc{R}_K^m g$ has the optimal convergence rate on the element $K \in
\MTh$ if $\overline{\Lambda}_m$ admits a uniform upper bound. 
For each $K \in \MTh$, we let $B_{r_K}$ and $B_{R_K}$ be the largest
ball and the smallest ball such that $B_{r_K} \subset \bigcup_{K' \in
S(K)} \overline{K'} \subset B_{R_K}$, with the radius $r_K$ and $R_K$,
respectively. By \cite[Lemma 5]{Li2016discontinuous}, there holds
$\Lambda_m \leq 1 +\varepsilon$ under the condition $r_K \geq m
\sqrt{2R_K h_K(1 + 1/\varepsilon)}$. Generally speaking, if the patch
$S(K)$ is large enough, there will be $r_K \approx R_K$, and this
condition can be then satisfied.
Therefore, for this condition we are
required to construct a wide element patch to ensure a large $r_K$. 
In \cite[Lemma 5]{Li2016discontinuous}\cite[Lemma
3.4]{Li2012efficient}, we prove that 
this condition can be met when the threshold $\# S$ in
\eqref{eq_patch} is greater than a certain constant, which only
depends on $m$ and $C_\sigma$.  We also notice that this given bound of
the size of the patch is usually too large and is impractical in the
computer implementation. As
numerical observations, the reconstructed method works very well when
the size of the patch is far less than the theoretical value. In
Section \ref{sec_numericalresults}, we list the values $\# S$ in all
tests. The threshold is roughly taken as $\# S \approx
\frac{d+1}{2} \dim(\mb{P}_m(\cdot))$ on quasi-uniform meshes.

\substitute{}{
\begin{remark}
  For the patch $S(K)$, we consider the special case that the
  corresponding collocation points set $I(K)$ has exactly
  $\dim(\mb{P}_m(\cdot))$ points, then the solution to the least
  squares problem \eqref{eq_leastsquares} becomes the Lagrange
  interpolation polynomial and the constant $\Lambda(m, K)$ is equal
  to the Lebesgue constant \cite{Powell1981approximation}. To our best
  knowledge, in two and three dimensions there are few results 
  about the bound of the Lebesgue constant. 
  The Lebesgue constant may grow very fast as $h$ tends to zero, which
  will hamper the convergence of the scheme. 
  Currently, we can prove that the constant $\Lambda_m$ admits a
  uniform upper bound for the wider element patch.  The wider element
  patch will bring more computational cost for filling the stiffness
  matrix and increase the width of the banded structure, which also
  leads to more computational cost in the matrix-vector product for
  our method.
  This can be regarded as the price we pay for the upper bound to
  $\Lambda_m$. 
  On the other hand, we can give a preconditioning method based on the
  piecewise constant space for any high-order accuracy, which can be
  computed efficiently. Consequently, solving the resulting linear
  system is still observed to be fast in the numerical tests.

  Furthermore, the upper bound of $\overline{\Lambda}_m$ certainly
  depends on the polynomial degree $m$, and $\overline{\Lambda}_m$
  will grow larger as $m$ increases. The precise dependence between
  $\overline{\Lambda}_m$ and $m$ and the $h$-$m$ version of the
  reconstructed space are considered in the future research.
  \label{re_Lambda}
\end{remark}
}


\section{Approximation to the elliptic problem}
\label{sec_elliptic}
In this section, we present the numerical scheme for the elliptic
problem \eqref{eq_elliptic}, based on the interior penalty method and
the reconstructed space $U_h^m$.  We seek $u_h \in U_h^m$ such that
\begin{equation}
  a_{h, \theta}(u_h, v_h) = l_{h, \theta}(v_h), \quad \forall v_h \in
  U_h^m,
  \label{eq_dvar}
\end{equation}
where
\begin{equation}
  \begin{aligned}
    a_{h, \theta}(u_h, v_h) := \sum_{K \in \MTh} \int_K A\nabla u_h \cdot \nabla
    v_h \d{\bm{x}} &-\sum_{e \in \MEh} \int_e \aver{A\nabla u_h} \cdot
    \jump{v_h} \d{\bm{s}} \\
    & + \theta \sum_{e \in \MEh} \int_e \aver{A\nabla v_h}
    \cdot \jump{u_h} \d{\bm{s}} + \sum_{e \in \MEh} \int_e {\mu}
    h_e^{-1} \jump{u_h} \jump{v_h} \d{\bm{s}},
  \end{aligned}
  \label{eq_ah}
\end{equation}
and
\begin{displaymath}
  l_{h, \theta}(v_h) := \sum_{K \in \MTh} \int_K f_h v_h \d{\bm{x}} + \theta
  \sum_{e \in \MEhB} \int_e g \aver{A\nabla v_h} \d{\bm{s}} + \sum_{e
  \in \MEhB} \int_e \mu h_e^{-1}  g \jump{v_h} \d{s}.
\end{displaymath}
Here $\mu$ is the penalty parameter.  We refer to
\cite{Arnold2002unified,Arnold1982interior} for the derivation of the
bilinear form for the interior penalty method.  The forms with $\theta
= -1/1$ are known as the symmetric/nonsymmetric interior penalty DG
methods. 

The convergence analysis follows from the standard procedure under the
Lax-Milgram framework. Let $U_h := U_h^m + H^2(\Omega)$, and we define
the following \substitute{energy norms}{mesh-dependent energy norms}
for the error estimation, which read
\begin{displaymath}
  \DGenorm{v_h}^2 := \sum_{K \in \MTh} \| \nabla v_h \|_{L^2(K)}^2 +
  \sum_{e \in \MEh} h_e^{-1} \| \jump{v_h} \|^2_{L^2(e)}, \quad \forall
  v_h \in U_h, 
\end{displaymath}
and
\begin{displaymath}
  \DGtenorm{v_h}^2 := \DGenorm{v_h}^2 + \sum_{e \in \MEh} h_e \|
  \aver{\nabla v_h} \|^2_{L^2(e)}, \quad \forall v_h \in U_h.
\end{displaymath}
It is noticeable that both norms are equivalent restricted on the
piecewise polynomial space $U_h^m$, i.e.
\begin{equation}
  \DGenorm{v_h} \leq \DGtenorm{v_h} \leq C \DGenorm{v_h}, \quad
  \forall v_h \in U_h^m.
  \label{eq_DGnorm_equi}
\end{equation}
The above equivalence estimate follows from the trace estimate
\eqref{eq_trace} and the inverse estimate \eqref{eq_inverse}. In
addition, 
from \cite[Lemma 2.1]{Arnold1982interior}, we give the relationship
between the energy norm and the $L^2$ norm, which reads
\begin{equation}
  \| v_h \|_{L^2(\Omega)} \leq C \DGenorm{v_h} \leq Ch^{-1} \| v_h
  \|_{L^2(\Omega)}, \quad \forall v_h \in U_h^m.
  \label{eq_DGnormL2bound}
\end{equation}
Because the
reconstructed space $U_h^m$ is a subspace of the standard
discontinuous piecewise polynomial space, 
\substitute{the above steps are standard in the DG framework
\cite{Arnold2002unified}}{the following steps, including the
boundedness and the coercivity of $a_{h, \theta}(\cdot, \cdot)$ and
the Galerkin orthogonality, can be derived with a standard procedure
in the DG framework \cite{Arnold2002unified}.}
\begin{lemma}
  For the symmetric method $\theta = -1$, we let $a_{h, \theta}(\cdot,
  \cdot)$ be defined with a sufficiently large $\mu$, and for the
  nonsymmetric method $\theta = 1$, we let $a_{h, \theta}(\cdot,
  \cdot)$ be defined with any positive $\mu > 0$, there exist
  constants $C$ such that
  \begin{align}
    |a_{h, \theta}(v_h, w_h)| & \leq C \DGtenorm{v_h} \DGtenorm{w_h},
    \quad \forall v_h, w_h \in U_h, \label{eq_ahb} \\
    a_{h, \theta}(v_h, v_h) & \geq C \DGtenorm{v_h}^2, \quad \forall
    v_h \in U_h^m.
    \label{eq_ahc}
  \end{align}
  \label{le_ahbc}
\end{lemma}
\substitute{}{We can also share the estimation of the penalty
parameter in the symmetric interior penalty DG method, i.e. there
exists a threshold $\mu_0(m, C_\sigma, A)$ such that $\mu \geq \mu_0$
yields the coercivity \eqref{eq_ahc}. We refer to \cite[Remark 12 and
Remark 17]{Epshteyn2007estimation} for more details on $\mu_0$, and we
also list the choice of $\mu$ in numerical tests in Section
\ref{sec_numericalresults}.
}
\begin{lemma}
  Let $u \in H^2(\Omega)$ be the exact solution to
  \eqref{eq_elliptic}, and let $u_h \in U_h^m$ be the numerical
  solution to \eqref{eq_dvar}, there holds
  \begin{equation}
    a_{h, \theta}(u - u_h, v_h) = 0, \quad \forall v_h \in U_h^m.
    \label{eq_ahorth}
  \end{equation}
  \label{le_ahorth}
\end{lemma}
Combining the approximation result \eqref{eq_Rmapp} of the space
$U_h^m$ and Lemma \ref{le_ahbc} - Lemma \ref{le_ahorth} yields the
desired error estimation.
\begin{theorem}
  Let $u \in H^{m+1}(\Omega)$ be the exact solution to
  \eqref{eq_elliptic}, and let $u_h \in U_h^m$ be the numerical
  solution to \eqref{eq_dvar}, and let the penalty parameter $\mu$ be
  taken as in Lemma \ref{le_ahbc}, there exists a constant $C$ such
  that
  \begin{equation}
    \DGenorm{u - u_h} \leq C \lLambda_m h^m \| u \|_{H^{m+1}(\Omega)}.
    \label{eq_DGerror}
  \end{equation}
  In addition, for the symmetric method $\theta = -1$, there holds
  \begin{equation}
    \|u - u_h \|_{L^2(\Omega)} \leq C \lLambda_m h^{m+1} \|
    u \|_{H^{m+1}(\Omega)}.
    \label{eq_L2error}
  \end{equation}
  \label{th_ellerror}
\end{theorem}
\begin{proof}
  From Lemma \ref{le_ahbc} - Lemma \ref{le_ahorth}, one can find
  that
  \begin{displaymath}
    \DGtenorm{u - u_h} \leq C \DGtenorm{u - v_h}, \quad\forall v_h \in
    U_h^m.
  \end{displaymath}
  Let $v_h := \mc{R}^m u$. 
  By the approximation property \eqref{eq_Rmapp} and the trace
  estimate \eqref{eq_trace}, we have that
  \begin{displaymath}
    \DGtenorm{u - v_h} \leq C{\lLambda}_m h^m \| u
    \|_{H^{m+1}(\Omega)}.
  \end{displaymath}
  Collecting the above two estimates gives the error estimate
  \eqref{eq_DGerror}. 

  For the symmetric scheme with $\theta = -1$, 
  the $L^2$ error estimation \eqref{eq_L2error} can be obtained by the
  standard dual argument, which completes the proof.
\end{proof}

The error estimates for the nonsymmetric/symmetric
interior penalty DG methods have been completed. 
From Theorem \ref{th_ellerror}, the numerical solution has the optimal
convergence rates, as the standard finite element methods. In
addition, the proposed method also shares the flexibility in the 
mesh partition, i.e. the polygonal elements are allowed in the scheme.
Since the number of degrees of freedom is fixed as $n_e$ for any $m
\geq 1$, one of the advantages is that the scheme with $U_h^m$ will enjoy
a better efficiency on the finite element approximation, see
Subsection \ref{subsec_convergence} for a numerical comparison on the
proposed method and the standard DG method.

Next, we focus on the resulting linear system arising in our method.
We begin by estimating the condition number. 
Let $A_{m, \theta}$ be the linear system with respect to the bilinear
form $a_{h, \theta}(\cdot, \cdot)$ of order $m$. As commented in Section
\ref{sec_space}, $U_h^m$ is spanned by the group of basis functions
$\{ \lambda_j \}_{j = 1}^{n_e}$. Hence, $A_{m, \theta}$ can be
expressed
as $A_{m, \theta} = (a_{h,\theta}(\lambda_i, \lambda_j))_{n_e \times
n_e}$. We define $M_m := (\lambda_i, \lambda_j)_{n_e \times n_e}$ as
the mass matrix. For any $\bmr{v} = \{ v_j\}_{j=1}^{n_e} \in
\mb{R}^{n_e}$, we can associate $\bmr{v}$ with a finite element
function $v_h \in U_h^m$ such that $v_h = \sum_{j =1}^{n_e} v_j
\lambda_j$, i.e. $\bmr{v}$ is the vector of coefficients in
\eqref{eq_Rmvexpansion} for $v_h$. 
Conversely, any $v_h \in U_h^m$ has the decomposition $v_h =
\sum_{j=1}^{n_e} v_h(\bm{x}_{K_j})\lambda_j$, and we define 
$\bmr{v} :=\{ v_h(\bm{x}_{K_j}) \}_{j=1}^{n_e} \in
\mb{R}^{n_e}$ as the corresponding vector for $v_h$. We have that
\begin{displaymath}
  a_{h, \theta}(v_h, v_h) = \bmr{v}^T A_{m, \theta} \bmr{v}, \quad (v_h,
  v_h)_{L^2(\Omega)} = \bmr{v}^T M_m \bmr{v}, \quad \forall v_h \in
  U_h^m.
\end{displaymath}
Let us bound the term $(\bmr{v}^T M_m \bmr{v})/(\bmr{v}^T \bmr{v})$,
which is established on the constraint in \eqref{eq_leastsquares}. 
From \eqref{eq_leastsquares} and the inverse estimate
\eqref{eq_inverse}, we find that
\begin{equation}
  \begin{aligned}
    \bmr{v}^T\bmr{v} = \sum_{K \in \MTh} (v_h(\bm{x}_K))^2 \leq
    \sum_{K \in \MTh} \| v_h \|_{L^{\infty}(K)}^2 \leq Ch^{-d} \sum_{K
    \in \MTh} \| v_h \|_{L^2(K)}^2 = Ch^{-d} \bmr{v}^T M_m \bmr{v},
    \quad \forall \bmr{v} \in \mb{R}^{n_e}.
  \end{aligned}
  \label{eq_vTvupper}
\end{equation}
Again by the constraint in \eqref{eq_leastsquares}, we deduce that
\begin{equation}
    \bmr{v}^T M_m \bmr{v} 
    \leq Ch^d \sum_{K \in \MTh} \| v_h
    \|_{L^{\infty}(K)}^2  
    \leq Ch^d \Lambda_m^2 \sum_{K \in \MTh} \max_{\bm{x} \in I(K)} (
    v_h(\bm{x}))^2 \leq Ch^d {\overline{\Lambda}}_m^2 \bmr{v}^T
    \bmr{v}, \quad \forall \bmr{v} \in \mb{R}^{n_e}.
  \label{eq_vTvlower}
\end{equation}
Combining \eqref{eq_vTvupper} and \eqref{eq_vTvlower} yields that
\begin{equation}
  Ch^d \leq \frac{\bmr{v}^T M_m \bmr{v} }{ \bmr{v}^T \bmr{v}} \leq
  C{\overline{\Lambda}}_m^2 h^d, \quad \forall  \bmr{v} \neq \bmr{0} \in
  \mb{R}^{n_e}.
  \label{eq_vTMv}
\end{equation}
For the symmetric case $\theta = -1$, we have that
\begin{equation}
  \frac{\bmr{v}^T A_{m, \theta} \bmr{v}}{\bmr{v}^T \bmr{v}} =
  \frac{a_{h,\theta}( {v}_h,
   {v}_h)}{ ({v}_h, {v}_h)_{L^2(\Omega)} }
   \frac{\bmr{v}^T M_m \bmr{v}}{ \bmr{v}^T \bmr{v}}, \quad \forall
   \bmr{v} \neq \bmr{0} \in \mb{R}^{n_e}
   \label{eq_vTAv}
\end{equation}
Collecting the boundedness \eqref{eq_ahb}, the coercivity
\eqref{eq_ahc} and the estimate \eqref{eq_DGnormL2bound} immediately
brings us that 
\begin{displaymath}
  \| v_h \|_{L^2(\Omega)}^2 \leq C a_{h, \theta}(v_h, v_h) \leq C
  h^{-2} \| v_h \|_{L^2(\Omega)}^2, \quad \forall v_h \in U_h^m.
\end{displaymath}
From \eqref{eq_vTMv}, there holds
\begin{equation}
  Ch^d \leq \frac{\bmr{v}^TA_{m, \theta} \bmr{v} }{\bmr{v}^T \bmr{v}}
  \leq Ch^{d-2} {\overline{\Lambda}}_m^2, \quad \forall 
  \bmr{v} \neq \bmr{0} \in \mb{R}^{n_e},
  \label{eq_Amtheta}
\end{equation}
which indicates that $\kappa(A_{m, \theta}) \leq
C{\overline{\Lambda}}_m^2 h^{-2}$ for the symmetric case.


Next, we turn to the nonsymmetric case, i.e. $\theta = 1$. We split
the bilinear form $a_{h, \theta}(\cdot, \cdot)$ into a symmetric part
and an antisymmetric part. We define the following bilinear forms 
\begin{displaymath}
  a_{h, \theta}^S(v_h, w_h) := \sum_{K \in \MTh} \int_K A\nabla v_h
  \cdot \nabla w_h \d{x} + \sum_{e \in \MEh} \int_e \mu h_e^{-1}
  \jump{v_h} \cdot \jump{w_h} \d{s}, \quad \forall v_h, w_h \in U_h^m
\end{displaymath}
and
\begin{displaymath}
  a_{h, \theta}^N(v_h, w_h) := \sum_{e \in \MEh} \int_e \aver{A\nabla
  v_h} \cdot \jump{w_h} \d{s} -  \sum_{e \in \MEh} \int_e \aver{A\nabla
  w_h} \cdot \jump{v_h} \d{s}, \quad \forall v_h, w_h \in U_h^m.
\end{displaymath}
Clearly, there holds $a_{h, \theta}(v_h, w_h) = a_{h, \theta}^S(v_h,
w_h) - a_{h, \theta}^N(v_h, w_h)$ for $\forall v_h, w_h \in U_h^m$.
We denote by $A_{m, \theta}^S$ and by $A_{m, \theta}^N$ the
symmetric/antisymmetric linear
systems \justsubstitute{with respect to}{to} $a_{h, \theta}^S(\cdot, \cdot)$ and $a_{h,
\theta}^N(\cdot, \cdot)$, respectively. We have that $A_{m, \theta} =
A_{m, \theta}^S - A_{m, \theta}^N$. From
Lemma \ref{le_ahbc}, it can be easily seen that 
\begin{equation}
  \begin{aligned}
    |a_{h, \theta}^{S}(v_h, w_h)| + |a_{h, \theta}^{N}(v_h, w_h)| &\leq
    C \DGenorm{v_h}\DGenorm{w_h}, && \forall v_h, w_h \in U_h^m, \\
    a_{h, \theta}^S(v_h, v_h) & \geq C
    \substitute{\DGenorm{v_h}}{\DGenorm{v_h}^2}, && \forall
    v_h \in U_h^m.
  \end{aligned}
  \label{eq_aSaNbc}
\end{equation}
Similar to $A_{m, \theta}$ with $\theta = -1$, the symmetric part
$A_{m, \theta}^S$ also satisfies the estimate \eqref{eq_Amtheta},
which gives the estimate $\kappa(A_{m, \theta}^S) \leq C
{\overline{\Lambda}}_m^2 h^{-2}$.

For any $\bmr{v}, \bmr{w} \in \mb{R}^{n_e}$ with $\| \bmr{v} \|_{l^2}
= \| \bmr{w} \|_{l^2} = 1$, we have that 
\begin{displaymath}
  \begin{aligned}
    \bmr{v}^T A_{h, \theta}^N \bmr{w} & = a_{h, \theta}^N(v_h, w_h)
    \leq C \DGenorm{v_h}\DGenorm{w_h} \leq Ch^{-2} \| v_h
    \|_{L^2(\Omega)} \| w_h \|_{L^2(\Omega)} \\
    & \leq C h^{-2} (\bmr{v}^T M_m \bmr{v})^{1/2} (\bmr{w}^T M_m
    \bmr{w})^{1/2} \leq C  {\overline{\Lambda}}_m^2 h^{d-2}.
  \end{aligned}
\end{displaymath}
Since $A_{m, \theta}^N$ is antisymmetric, all eigenvalues are purely
imaginary. Then, the spectral radius of $A_{m, \theta}^N$ satisfies
that $\rho(A_{m, \theta}^N) \leq  C {\overline{\Lambda}}_m^2h^{d-2}$.
From $A_{m, \theta} = A_{m, \theta}^S - A_{m, \theta}^N$, we conclude
that $\kappa(A_{m, \theta}) \leq C{\overline{\Lambda}}_m^2 h^{-2}$.

Consequently, we have the following estimate to the upper bound of the
condition number.
\begin{theorem}
  Let the penalty parameter $\mu$ be defined as in Lemma
  \ref{le_ahbc}, there exists a constant $C$ such that
  \begin{equation}
    \kappa(A_{m, \theta}) \leq C {\overline{\Lambda}}_m^2 h^{-2},
    \quad \theta = \pm 1.
    \label{eq_kappaA}
  \end{equation}
  \label{th_kappaA}
\end{theorem}

By \eqref{eq_kappaA}, the condition number is the same as in the 
standard finite element method, i.e. $O(h^{-2})$. As $h$ tends to $0$,
the matrix $A_{m, \theta}$ becomes ill-conditioned. Hence, an
effective preconditioner is desired in finite element methods
especially for the high-order accuracy. 
In the reconstructed approximation, one of attractive features 
is that the
matrix $A_{m, \theta}$ always has the size $n_e \times n_e$,
independent of the degree $m$. 
Consider the matrix that corresponds to the bilinear form $a_{h,
\theta}(\cdot, \cdot)$ over the piecewise constant spaces $U_h^0
\times U_h^0$, and clearly, this matrix still has the size $n_e \times
n_e$. We will show that this matrix can be used as an effective
preconditioner.  We define \justsubstitute{a}{the} bilinear form
$a_h^0(\cdot, \cdot)$ as 
\begin{equation}
  a_h^0(v_h, w_h) = \sum_{e \in \MEh} \int_e h_e^{-1} \jump{v_h} \cdot
  \jump{w_h} \d{s}, \quad \forall v_h, w_h \in U_h^0,
  \label{eq_ah0}
\end{equation}
which corresponds to the bilinear form $a_{h, \theta}(\cdot, \cdot)$
over the spaces $U_h^0 \times U_h^0$ in the sense that 
\justsubstitute{there holds}{}
$a_{h, \theta}(v_h, w_h) = \mu a_h^0(v_h, w_h)$ for $\forall v_h, w_h
\in U_h^0$ and $\theta = \pm 1$. Because $a_h^0(v_h, v_h) =
\DGenorm{v_h}^2$ for $\forall v_h \in U_h^0$, $a_h^0(\cdot, \cdot)$ is
bounded and coercive under the energy norm $\DGenorm{\cdot}$.
Let $A_0$ be the matrix of $a_h^0(\cdot, \cdot)$.
Clearly, $A_0$ is a symmetric positive definite matrix, and we will
show that 
\justsubstitute{$A_0^{-1}$ is an effective preconditioner such that}
$\kappa(A_0^{-1}A_{m, \theta}) \leq C{\overline{\Lambda}}_m^2$ for
both $\theta = \pm 1$.

For any piecewise constant function $v_h \in U_h^0$, we know that
$\mc{R}^m v_h \in U_h^m$ is a piecewise $m$-th degree polynomial
function after the reconstruction. The following equivalence
between $\DGenorm{v_h}$ and $\DGenorm{\mc{R}^m v_h}$ is crucial in the
analysis.
\begin{lemma}
  There exist constants $C$ such that 
  \begin{equation}
    \DGenorm{v_h} \leq C \DGenorm{\mc{R}^m v_h} \leq C
    {\overline{\Lambda}}_m \DGenorm{v_h},
    \quad \forall v_h \in U_h^0.
    \label{eq_DGenormCequi}
  \end{equation}
  \label{le_DGenormCequi}
\end{lemma}
\begin{proof}
  We first prove the lower bound in \eqref{eq_DGenormCequi}, i.e. 
  \begin{displaymath}
    \sum_{e \in \MEh} h_e^{-1} \| \jump{v_h} \|_{L^2(e)}^2 \leq C (
    \sum_{K \in \MTh} \| \nabla \mc{R}^m v_h \|_{L^2(K)}^2 + \sum_{e
    \in \MEh} h_e^{-1} \| \jump{\mc{R}^m v_h} \|_{L^2(e)}^2), \quad
    \forall v_h \in U_h^0.
  \end{displaymath}
  \justsubstitute{For any interior $e \in \MEhI$ shared by $K_0, K_1$,
  i.e. $e = \partial K_0 \cap \partial K_1$.}{For any interior face $e
  \in \MEhI$, we let $e$ be shared by two elements $K_0, K_1$, i.e. $e
  = \partial K_0 \cap \partial K_1$.}
  From the constraint in \eqref{eq_leastsquares},  we know that
  $(\mc{R}^m v_h)(\bm{x}_{K_0}) = (\mc{R}_{K_0}^m v_h)(\bm{x}_{K_0}) =
  v_h(\bm{x}_{K_0})$ and $(\mc{R}^m v_h)(\bm{x}_{K_1}) =
  (\mc{R}_{K_1}^m v_h)(\bm{x}_{K_1}) = v_h(\bm{x}_{K_1})$.
  Let $\bm{x}_e$ be any point on $e$, together with the inverse
  estimate \eqref{eq_inverse}, we deduce that
  \begin{align*}
    h_e^{-1} &\| \jump{v_h} \|_{L^2(e)}^2  \leq C h_e^{d-2}
    (v_h(\bm{x}_{K_0}) - v_h(\bm{x}_{K_1}))^2 = C h_e^{d-2} (
    (\mc{R}_{K_0}^m v_h)(\bm{x}_{K_0}) - (\mc{R}_{K_1}^m
    v_h)(\bm{x}_{K_1}) )^2 \\
    & \leq  C h_e^{d-2} ( ( (\mc{R}_{K_0}^m v_h)(\bm{x}_{K_0}) -
    (\mc{R}_{K_0}^m v_h)(\bm{x}_e)  )^2 \\
    & \hspace{50pt} + ( (\mc{R}_{K_1}^m
    v_h)(\bm{x}_{K_1}) - (\mc{R}_{K_1}^m v_h)(\bm{x}_e) )^2 + (
    (\mc{R}_{K_0}^m v_h)(\bm{x}_e) - (\mc{R}_{K_1}^m v_h)(\bm{x}_e)
    )^2) \\
    & \leq Ch_e^{d-2}( h_{K_0}^2 \| \nabla \mc{R}_{K_0}^m v_h
    \|_{L^{\infty}(K_0)}^2 + h_{K_1}^2 \| \nabla \mc{R}_{K_1}^m v_h
    \|_{L^{\infty}(K_1)}^2  + \| \jump{\mc{R}^m v_h}
    \|_{L^{\infty}(e)}^2) \\
    & \leq C ( \| \nabla \mc{R}_{K_0}^m v_h \|_{L^{2}(K_0)}^2  + 
    \| \nabla \mc{R}_{K_1}^m v_h \|_{L^{2}(K_1)}^2 + h_e^{-1} \|
    \jump{\mc{R}^m v_h} \|_{L^{2}(e)}^2).
  \end{align*}
  For any boundary face $e \in \MEhB$, we let $e$ be a face of an
  element $K_e$. Similarly, we have that $(\mc{R}^m v_h)(\bm{x}_{K_e})
  = (\mc{R}_{K_e}^m v_h)(\bm{x}_{K_e}) = v_h(\bm{x}_{K_e})$. Let
  $\bm{x}_e$ be any point on $e$, we have that
  \begin{align*}
    h_e^{-1} \| v_h ||_{L^2(e)}^2 & = h_e^{d-2}
    (v_h(\bm{x}_{K_e}))^2 = h_e^{d-2}
    ( (\mc{R}_{K_e}^m v_h)(\bm{x}_{K_e}))^2 \\
    & \leq  C h_e^{d-2} ( ( (\mc{R}_{K_e}^m v_h)(\bm{x}_{K_e}) -
    (\mc{R}_{K_e}^m v_h)(\bm{x}_e))^2 +  ((\mc{R}_{K_e}^m
    v_h)(\bm{x}_e))^2) \\
    & \leq C ( \| \nabla \mc{R}_{K_e}^m v_h \|_{L^2(K_e)}^2 + h_e^{-1}
    \| \jump{\mc{R}^m_{K_e} v_h} \|_{L^2(e)}^2).
  \end{align*}
  Summation over all faces yields that $\DGenorm{v_h} \leq C
  \DGenorm{\mc{R}^m v_h}$. The lower bound of \eqref{eq_DGenormCequi}
  is reached.

  Then, we focus on the upper bound, i.e.
  \begin{displaymath}
    \sum_{K \in \MTh} \| \nabla \mc{R}^m v_h \|_{L^2(K)}^2 + \sum_{e
    \in \MEh} h_e^{-1} \| \jump{\mc{R}^m v_h} \|_{L^2(e)}^2 \leq C
    {\overline{\Lambda}}_m^2 \sum_{e \in \MEh} h_e^{-1} \| \jump{v_h}
    \|_{L^2(e)}^2, \quad \forall v_h \in U_h^0.
  \end{displaymath}
  For $v_h \in U_h^0$, we let
  $v_{K, \max} := \max_{\bm{x} \in I(K)} v_h(\bm{x})$ and $v_{K, \min}
  := \min_{\bm{x} \in I(K)} v_h(\bm{x})$ for $\forall K \in \MTh$.
  We apply the inverse estimate \eqref{eq_inverse} and Lemma
  \ref{le_recons_stability} to see that
  for $\forall K \in \MTh$, there holds
  \begin{align}
    \| \nabla \mc{R}^m v_h \|_{L^2(K)}^2 &=  \| \nabla \mc{R}^m_K v_h
    \|_{L^2(K)}^2  = \| \nabla (\mc{R}^m_K v_h - v_{K, \min})
    \|_{L^2(K)}^2 \leq C h_K^{d - 2} \| \mc{R}_K^m ( v_h - v_{K,
    \min}) \|_{L^{\infty}(K)}^2  \nonumber \\
    & \leq  Ch_K^{d-2} \overline{\Lambda}_m^2 \max_{\bm{x} \in I(K)} |
    v_h(\bm{x}) - v_{K, \min} |^2 \leq Ch_K^{d-2}
    {\overline{\Lambda}}_m^2 (v_{K, \max} - v_{K, \min})^2. 
    \label{eq_nablavhupper}
  \end{align}
  \justsubstitute{Let $e \in \MEhI$ be shared by $K_0$ and $K_1$, we
  derive that}{For any interior face $e \in \MEhI$, we also let $e$ be
  shared by $K_0$ and $K_1$, and we derive that}
  \begin{align*}
    & h_e^{-1}\| \jump{\mc{R}^m v_h} \|_{L^2(e)}^2 = h_e^{-1} \|
    \mc{R}_{K_0}^m v_h - \mc{R}_{K_1}^m v_h \|_{L^2(e)}^2 \\
    &= h_e^{-1} \| \mc{R}_{K_0}^m v_h - (\mc{R}_{K_0}^m
    v_h)(\bm{x}_{K_0}) - \mc{R}_{K_1}^m v_h +  (\mc{R}_{K_1}^m
    v_h)(\bm{x}_{K_1}) + (\mc{R}_{K_0}^m
    v_h)(\bm{x}_{K_0}) -  (\mc{R}_{K_1}^m v_h)(\bm{x}_{K_1}) 
    \|_{L^2(e)}^2 \\
    & \leq Ch_e^{-1} (h_{K_0}^{d+1} \| \nabla \mc{R}_{K_0}^m
    v_h\|^2_{L^{\infty}(K_0)} +
    h_{K_1}^{d+1} \| \nabla \mc{R}_{K_1}^m v_h\|^2_{L^{\infty}(K_1)})
    + Ch_e^{d-2} (v_h(\bm{x}_{K_0}) - v_h(\bm{x}_{K_1}))^2 \\
    & \leq C ( \| \nabla \mc{R}_{K_0}^m
    v_h\|^2_{L^{2}(K_0)} + \| \nabla \mc{R}_{K_1}^m
    v_h\|^2_{L^{2}(K_1)} + h_e^{-1} \| \jump{v_h} \|_{L^2(e)}^2).
  \end{align*}
  For any boundary face $e \in \MEhB$, we let $e$ be a face of an
  element $K_e$, and there holds
  \begin{displaymath}
    \begin{aligned}
      h_e^{-1} \| \jump{\mc{R}^m v_h} \|_{L^2(e)}^2 = h_e^{-1} \|
      \mc{R}^m_{K_e} v_h \|_{L^2(e)}^2 \leq C ( \| \nabla \mc{R}^m_{K_e} 
      v_h \|_{L^2(K_e)}^2 + h_e^{-1} \| \jump{v_h} \|_{L^2(e)}^2).
    \end{aligned}
  \end{displaymath}
  Collecting above two estimates brings us that
  \begin{equation}
    \DGenorm{\mc{R}^m v_h}^2 \leq C  \overline{\Lambda}_m^2 \sum_{K
    \in \MTh} h_K^{d-2} (v_{K, \max} - v_{K, \min})^2 + C
    \DGenorm{v_h}^2.
    \label{eq_DGnrmRmvh}
  \end{equation}
  For $\forall K \in \MTh$, we let $\mc{E}_{S(K)} := \bigcup_{K' \in
  S(K)} \mc{E}_{K'}$, where $\mc{E}_{K'} := \{ e \in \MEh \ | \ e
  \subset \partial K'\}$, be the set of all faces corresponding to all
  elements in $S(K)$.  We then show that for $\forall K \in \MTh$,
  there holds
  \begin{equation}
    (v_{K, \max} - v_{K, \min})^2 \leq C\sum_{e \in \mc{E}_{S(K)}}
    h_e^{1-d} \| \jump{v_h} \|_{L^2(e)}^2.
    \label{eq_vmaxmin}
  \end{equation}
  Let $K', K'' \in S(K)$ such that $v_{K, \max} = v_h|_{K'}$ and
  $v_{K, \min} = v_h|_{K''}$. Clearly, there exists a sequence
  $\wt{K}_0, \wt{K}_1, \ldots, \wt{K}_M$ such that $\wt{K}_0 = K'$,
  $\wt{K}_M = K''$, $\wt{K}_j \in S(K)(0 \leq j \leq M)$, and
  $\wt{K}_j$ is adjacent to $\wt{K}_{j+1}$, and we let
  $\wt{e}_j = \partial
  \wt{K}_j \cap \partial \wt{K}_{j+1}$. We have that
  \begin{align*}
    (v_{K, \max} - v_{K, \min})^2 \leq C \sum_{j = 0}^{M - 1}
    (v_h|_{K_j} - v_h|_{K_{j+1}})^2  \leq C \sum_{j = 0}^{M - 1}
    h_{\wt{e}_j}^{1-d} \| \jump{v_h} \|_{L^2(\wt{e}_j)}^2 \leq C
    \sum_{e \in \mc{E}_{S(K)}} h_e^{1-d} \| \jump{v_h} \|_{L^2(e)}^2.
  \end{align*}
  Combining \eqref{eq_DGnrmRmvh} and \eqref{eq_vmaxmin} leads
  to the upper bound in \eqref{eq_DGenormCequi}. This completes the
  proof.
\end{proof}
The condition number of the preconditioned system $A_0^{-1} A_{m,
\theta}$ will be established on Lemma \ref{le_DGenormCequi}. 
We begin the analysis from the symmetric case, i.e. $\theta = -1$.
In the estimation of \substitute{the condition number}{the condition
number to the matrix $A_{m,
\theta}$}, we stated that any vector
$\bmr{v} \in \mb{R}^{n_e}$ corresponds to a finite element function
$v_h \in U_h^m$. For the analysis to the preconditioned system, 
we let any $\bmr{v} = \{v_j\}_{j = 1}^{n_e} \in \mb{R}^{n_e}$
correspond to a piecewise constant function $v_h \in U_h^0$ such that
$v_h(\bm{x}_K) = v_j$, where $K$ is the element indexed by $j$. 
Now, we have that 
\begin{displaymath}
  \bmr{v}^T A_0 \bmr{w} = a_h^0(v_h, w_h), \quad \bmr{v}^T A_{m,
  \theta} \bmr{w} = a_{h, \theta}(\mc{R}^m v_h, \mc{R}^m w_h), \quad
  \forall \bmr{v}, \bmr{w} \in \mb{R}^{n_e}.
\end{displaymath}
By Lemma \ref{le_DGenormCequi} and the boundedness and the coercivity
of $a_{h, \theta}(\cdot, \cdot)$, we find that 
\begin{equation}
  C_0 {\overline{\Lambda}}_m^{-2} \bmr{v}^T A_{m, \theta} \bmr{v} \leq
  \bmr{v}^T A_{0} \bmr{v} \leq C_1 \bmr{v}^T A_{m, \theta} \bmr{v},
  \quad \forall \bmr{v} \in \mb{R}^{n_e}.
  \label{eq_LmvTAv}
\end{equation}
From \cite[Lemma 2.1]{Xu1992iterative}, the above estimate yields the
desired estimate $\kappa(A_0^{-1} A_{m, \theta}) \leq C
\overline{\Lambda}_m^2$
for the symmetric case $\theta = -1$.

Next, we consider the preconditioned system for the nonsymmetric
case $\theta = 1$. In this case, we have that
\begin{equation}
  \bmr{v}^T A_{h, \theta}^S \bmr{w} = a_{h, \theta}^S(\mc{R}^m v_h,
  \mc{R}^m w_h), \quad \bmr{v}^T A_{h, \theta}^N \bmr{w} = a_{h,
  \theta}^N(\mc{R}^m v_h, \mc{R}^m w_h), \quad \forall \bmr{v},
  \bmr{w} \in \mb{R}^{n_e}.
  \label{eq_vTASw}
\end{equation}
The matrix $A_0^{-1} A_{m, \theta}$ can be split as $A_0^{-1} A_{m,
\theta} = A_0^{-1} A_{m, \theta}^S - A_0^{-1}A_{m, \theta}^N$. From
\eqref{eq_vTASw}, the estimate \eqref{eq_LmvTAv} also holds for the
symmetric part $A_{m, \theta}^S$, which gives that $\kappa(A_0^{-1}
A_{m, \theta}^S) \leq C \overline{\Lambda}_m^2$. For the antisymmetric
part, we deduce that 
\begin{equation}
  \begin{aligned}
  \bmr{v}^T A_{h, \theta}^N \bmr{w} & = a_{h, \theta}^N(\mc{R}^m v_h,
  \mc{R}^m w_h) \leq C \DGenorm{\mc{R}^m v_h} \DGenorm{\mc{R}^m w_h}
  \leq C \overline{\Lambda}_m^2 \DGenorm{v_h} \DGenorm{w_h} \\
  &\leq C \overline{\Lambda}_m^2 \sqrt{\bmr{v}^T A_0
  \bmr{v}}\sqrt{\bmr{w}^T A_0 \bmr{w}},
\end{aligned}
  \label{eq_vTANw}
\end{equation}
which implies that the spectral radius of the system $A_0^{-1} A_{h,
\theta}^N$ satifies that $\rho(A_0^{-1} A_{h, \theta}^N) \leq C
\overline{\Lambda}_m^2$. 

Ultimately, we summarize the estimation of the condition number for
the preconditioned system in the following theorem.
\begin{theorem}
  Let $\mu$ be defined as in Lemma \ref{le_ahbc}, there exists a
  constant $C$ such that
  \begin{equation}
    \kappa(A_0^{-1} A_{m, \theta}) \leq C \overline{\Lambda}_m^2,
    \quad \theta  = \pm 1.
    \label{eq_A0Amcond}
  \end{equation}
  \label{th_A0Amcond}
\end{theorem}

\substitute{}{
\begin{remark}
  The estimate of the condition number is based on the
  boundedness and the coercivity of the bilinear form $a_{h,
  \theta}(\cdot, \cdot)$, i.e. 
  \begin{displaymath}
    \wh{C} \DGenorm{v_h}^2 \leq a_{h, \theta}(v_h, v_h) \leq \wt{C}
    \DGenorm{v_h}^2, \quad \forall  v_h \in U_h^m. 
  \end{displaymath}
  For the symmetric scheme, by \eqref{eq_vTAv} and \eqref{eq_LmvTAv}
  we obtain that the constants in \eqref{eq_kappaA} and
  \eqref{eq_A0Amcond} depend on $\wt{C} / \wh{C}$, i.e.
  $\kappa(A_{m, \theta}) \leq C_0 (\wt{C} / \wh{C})
  \overline{\Lambda}_m^2 h^{-2}$ and $\kappa(A_0^{-1} A_{m, \theta})
  \leq C_1 (\wt{C} / \wh{C}) \overline{\Lambda}_m^2$, where $C_0$ and
  $C_1$ are independent of $\mu$ and $A$. We refer to
  \cite{Epshteyn2007estimation} for the detailed analysis of the
  relationship between $\wh{C}$, $\wt{C}$ and the penalty parameter
  $\mu$ and the coefficient $A$.
  \label{re_constant}
\end{remark}
}

\substitute{}{
By Theorem \ref{th_A0Amcond}, 
the resulting linear system $A_{m, \theta} \bmr{x} = \bmr{b}$ can be
solved by using an iterative approach from the Krylov-subspace family
with the preconditioner $A_0^{-1}$. In numerical tests, the
preconditioned CG/GMRES methods are applied to the
symmetric/nonsymmetric linear systems, respectively.
For the nonsymmetric system, we present more details about the
convergence analysis to the GMRES solver. 
We define an inner product $(\bmr{v}, \bmr{w})_{A_0} := \bmr{v}^T A_0
\bmr{w}$ for $\forall \bmr{v}, \bmr{w} \in \mb{R}^{n_e}$ and its
induced norm by $\| \cdot \|_{A_0}$.
By \cite[Theorem 1]{Xu1992GMRES}, the convergence estimate of the
preconditioned GMRES method follows from the estimates
\begin{equation}
  \| A_0^{-1} A_{m, 1} \bmr{v} \|_{A_0} \leq C_0
  \overline{\Lambda}_m^2 \| \bmr{v} \|_{A_0}, \quad \forall \bmr{v}
  \in \mb{R}^{n_e},
  \label{eq_A0Am1bounded}
\end{equation}
and
\begin{equation}
  (A_0^{-1} A_{m, 1} \bmr{v}, \bmr{v})_{A_0} \geq C_1 (\bmr{v},
  \bmr{v})_{A_0}, \quad \forall \bmr{v} \in \mb{R}^{n_e}.
  \label{eq_A0Am1coer}
\end{equation}
From \eqref{eq_vTASw} and \eqref{eq_LmvTAv}, we derive that
\begin{align*}
  (A_0^{-1} A_{m, 1} \bmr{v}, \bmr{v})_{A_0} = \bmr{v}^T A_{m, 1}
  \bmr{v} = \bmr{v}^T A_{m, 1}^S \bmr{v} \geq C \bmr{v}^T
  A_0 \bmr{v} = C (\bmr{v}, \bmr{v})_{A_0}, \quad \forall \bmr{v} \in
  \mb{R}^{n_e},
\end{align*}
which gives the second estimate \eqref{eq_A0Am1coer}.
By letting $\bmr{x} = A_0^{1/2} \bmr{v}$ and $\bmr{y} = A_0^{1/2}
\bmr{w}$ in \eqref{eq_vTANw}, we find that
\begin{displaymath}
  \bmr{x}^T A_0^{-1/2} A_{m, 1}^N A_0^{-1/2} \bmr{y}^T  \leq C
  \overline{\Lambda}_m^2 \sqrt{ \bmr{x}^T \bmr{x}} \sqrt{ \bmr{y}^T \bmr{y}} ,
  \quad \forall \bmr{x}, \bmr{y} \in \mb{R}^{n_e},
\end{displaymath}
which leads to $\sigma_{\max}(A_0^{-1/2} A_{m, 1}^N A_0^{-1/2}) \leq
C \overline{\Lambda}_m^2$. From the triangle inequality, 
we have that 
\begin{align*}
  \| A_0^{-1} A_{m, 1} \bmr{v} \|_{A_0} \leq  \| A_0^{-1} A_{m, 1}^S
  \bmr{v} \|_{A_0} +   \| A_0^{-1} A_{m, 1}^N
  \bmr{v} \|_{A_0}, \quad \forall \bmr{v} \in \mb{R}^{n_e}.
\end{align*}
By \eqref{eq_LmvTAv} and \cite[Lemma 2.1]{Xu1992iterative}, there holds
\begin{align*}
  \| A_0^{-1} A_{m, 1}^S \bmr{v} \|_{A_0}^2 = \bmr{v}^T A_{m, 1}^S
  A_0^{-1} A_{m, 1}^S \bmr{v} \leq C \overline{\Lambda}_m^4 \bmr{v}^T A_0
  \bmr{v} = C \overline{\Lambda}_m^4 \| \bmr{v} \|_{A_0}^2, \quad
  \forall \bmr{v} \in \mb{R}^{n_e}.
\end{align*}
For the second term, we derive that 
\begin{align*}
  \| A_0^{-1} A_{m, 1}^N \bmr{v} \|_{A_0}^2 & = \bmr{v}^T (A_{m,
  1}^N)^T A_0^{-1} A_{m, 1}^N \bmr{v} \\
  & = \bmr{w}^T (A_0^{-1/2} A_{m, 1}^N A_0^{-1/2} )^T A_0^{-1/2} A_{m,
  1}^N A_0^{-1/2} \bmr{w}  \text{ (let $\bmr{w} = A_0^{1/2} \bmr{v}$)}
    \\
  & \leq (\sigma_{\max}(  A_0^{-1/2} A_{m, 1}^N A_0^{-1/2} ))^2
  \bmr{w}^T \bmr{w} \leq C \overline{\Lambda}_m^4 \bmr{v}^T A_0
  \bmr{v} = C \overline{\Lambda}_m^4 \| \bmr{v} \|_{A_0}^2, \quad
  \forall \bmr{v} \in \mb{R}^{n_e}.
\end{align*}
The estimate \eqref{eq_A0Am1bounded} is reached.
From \eqref{eq_A0Am1bounded} and \eqref{eq_A0Am1coer}, the
preconditioned GMRES method for solving the system $A_{m, 1} \bmr{x} =
\bmr{b}$ has the following convergence rate \cite[Theorem
2]{Xu1992GMRES}, 
\begin{displaymath}
  \| \bmr{r}_k \|_{A_0} \leq ( 1 - (C_1/(C_0
  \overline{\Lambda}_m^2))^2)^{k/2} \| \bmr{r}_0 \|_{A_0}.
\end{displaymath}
where $\bmr{r}_k = \bmr{b} - A_{m, 1} \bmr{x}_k$ is the residual at
the iteration step $k$.
}

\substitute{From Theorem \ref{th_A0Amcond}, the linear system
$A_{m, \theta}$ can be solved using an iterative approach from the
Krylov-subspace family, such as BiCGSTAB, GMRES, with the
preconditioner $A_0^{-1}$.
In the
iteration, we are required to solve a linear system of the form $A_0
\bmr{x} = \bmr{b}$.}{
In the Krylov iteration step, we are required to compute the
matrix-vector product $A_0^{-1} \bmr{y}$. Hence, a fast and accurate
method
to solve the linear system of the form $A_0 \bmr{z} = \bmr{y}$ is
desired in our scheme. 
}
On the triangular meshes, we outline a \substitute{geometrical
multigrid method}{$\mc{V}$-cycle multigrid method} for this system.
Let $\mc{T}_{1}, \mc{T}_2, \ldots, \mc{T}_r$ be a series of
successively refined meshes, i.e. $\mc{T}_{l+1}$ is created by
subdividing all of triangular (tetrahedral) elements in $\mc{T}_l$.
Let $U_l^0$ be the piecewise constant \justsubstitute{function}{space} on
the partition $\mc{T}_l$, and we have that 
\begin{displaymath}
  U_1^0 \subset U_2^0 \subset U_3^0 \ldots \subset U_r^0.
\end{displaymath}
We let $I_{k}^{k+1}: U_{k}^0 \rightarrow U_{k+1}^0$ be the canonical
prolongation operator, i.e. $I_{k}^{k+1} v_h = v_h(\forall v_h \in
U_{k}^0)$, and we let $I_{k+1}^k: U_{k+1}^0 \rightarrow U_k^0$ be the
transpose of $I_k^{k+1}$. Let $A_{k}^0$ be the matrix for the bilinear
form $a_h^0(\cdot, \cdot)$ over the spaces $U_k^0 \times U_k^0$. 
Then, the standard recursive structure of the multigrid algorithm is
depicted in Algorithm ~\ref{alg_multigrid}.
Consequently, the preconditioner in preconditioned CG/GMRES methods to
the linear system  $A_{m, \theta} \bmr{x} = \bmr{b}$ can be chosen as
Algorithm \ref{alg_multigrid}. 
\substitute{ Consequently, the linear system $A_{m, \theta} \bmr{x} =
\bmr{b}$ is solved by an Krylov iterative method, and the
preconditioner is selected as the multigrid method of $A_0^{-1}$ in
Algorithm ~\ref{alg_multigrid}. For the polygonal mesh, we solve the
system $A_0 \bmr{x} = \bmr{b}$ by the algebraic multigrid method in
the package HYPRE \cite{Falgout2006design}.
}{We also apply the algebraic multigrid method (BoomerAMG in
the package HYPRE \cite{Falgout2006design}) and the direct LU method
to approximate $A_0^{-1}$ to precondition the linear system. 
The three different preconditioners give a close convergence step for
CG/GMRES solvers in numerical results. 
The theoretical analysis for multigrid methods in approximating
$A_0^{-1}$ is considered as a future work for us.
}

\begin{algorithm}[t]
  \caption{\substitute{Multigrid Solver}{$\mc{V}$-cycle Multigrid
  Solver}, MGSolver($\bmr{x}_k, \bmr{b}_k, k$)}
  \label{alg_multigrid}
  \KwIn{ the initial guess $\bmr{x}_k$, the right hand side
  $\bmr{b}_k$, the level $k$;}
  \KwOut{the solution $\bmr{x}_k$;}
  \If{$k=1$}{
  $\bmr{x}_1 = (A_1^{0})^{-1} \bmr{b}_1$, $\bmr{x}_1$ is the solution
  obtained from a direct method; \\[1mm]

  return $\bmr{x}_1$;
  }

  \If{$k > 1$}{
  presmoothing step: Gauss-Seidel sweep on $A_k^{0} \bmr{x}_k =
  \bmr{b}_k$;

  error correction step: let $\bmr{y} = I_k^{k-1}(\bmr{b}_k - A_k^0
  \bmr{x}_k)$, and \substitute{}{let $\bmr{z}_0 = \bm{0}$, and let 
  \begin{displaymath}
    \bmr{z}_1 = \text{MGSolver}(\bmr{z}_{0}, \bmr{y}, k - 1);
  \end{displaymath}}
  set $\bmr{x}_k = \bmr{x}_k + I_{k-1}^k\bmr{z}_1$.

  postsmoothing step: Gauss-Seidel sweep on $A_k^{0} \bmr{x}_k =
  \bmr{b}_k$;

  return $\bmr{x}_k$;
  }
\end{algorithm}

\section{Numerical Results}
\label{sec_numericalresults}
In this section, a series of numerical experiments are conducted to
demonstrate the performance of the proposed method and the efficiency
of the preconditioning method.  The meshes we used in the following
examples are shown in Fig.~\ref{fig_mesh}.  The threshold $\# S$ used
in 
all tests are listed in Tab.~\ref{tab_S2} and Tab.~\ref{tab_S3}.
\substitute{}{For the symmetric scheme $\theta = -1$, we take the
penalty parameter $\mu = 3m^2 + 5$, and for the nonsymmetric scheme
$\theta = 1$, we fix the penalty parameter $\mu = 1$.  }

\begin{table}
  \centering
  \renewcommand\arraystretch{1.35}
  \scalebox{0.9}{
  \begin{tabular}{p{1.0cm}|p{3.6cm}|p{1cm}|p{1cm}|p{1cm}|p{1cm}}
    \hline\hline
    \multicolumn{2}{c|}{$m$} & 1 & 2 & 3 & 4 \\
    \hline
    \multirow{2}{*}{$\# S$} & the triangular mesh & 
    5 & 9 & 15 & 21 \\
    \cline{2-6}
    & the polygonal mesh & 6 & 10 & 16 & 23  \\
    \hline
    \multicolumn{2}{c|}{$\dim(\mb{P}_m)$} &  3 & 6 & 10 & 15 \\
    \hline\hline
  \end{tabular}}
  \caption{The threshold $\# S$ in two dimensions.}
  \label{tab_S2}
\end{table}

\begin{table}
  \centering
  \renewcommand\arraystretch{1.3}
  \scalebox{0.9}{
  \begin{tabular}{p{1.0cm}|p{3.6cm}|p{1cm}|p{1cm}|p{1cm}}
    \hline\hline
    \multicolumn{2}{c|}{$m$} & 1 & 2 & 3 \\
    \hline
    $\# S$  & the tetrahedral mesh & 9 & 19 & 38 \\
    \hline
    \multicolumn{2}{c|}{$\dim(\mb{P}_m)$} & 4 & 10 & 20 \\
    \hline\hline
  \end{tabular}}
  \caption{The threshold $\# S$ in three dimensions.}
  \label{tab_S3}
\end{table}

\begin{figure}[htbp]
  \centering
  \includegraphics[width=0.25\textwidth]{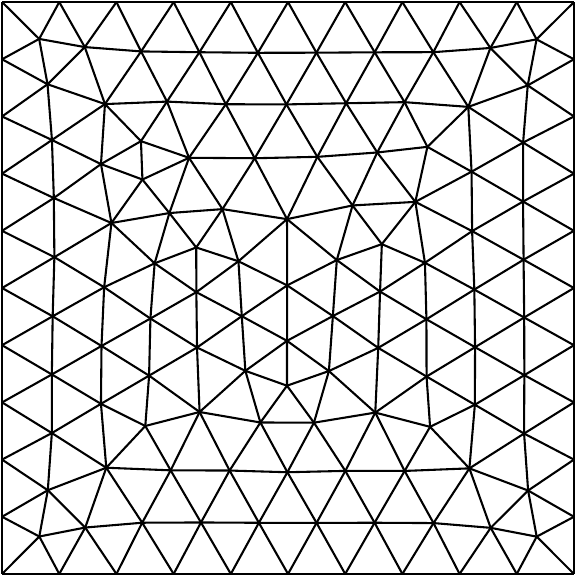}
  \hspace{30pt}
  \begin{minipage}[t]{0.25\textwidth}
    \centering
    \begin{tikzpicture}[scale = 2]
      \input{./figure/polymesh.tex}
    \end{tikzpicture}
  \end{minipage}
  \hspace{30pt}
  \includegraphics[width=0.25\textwidth]{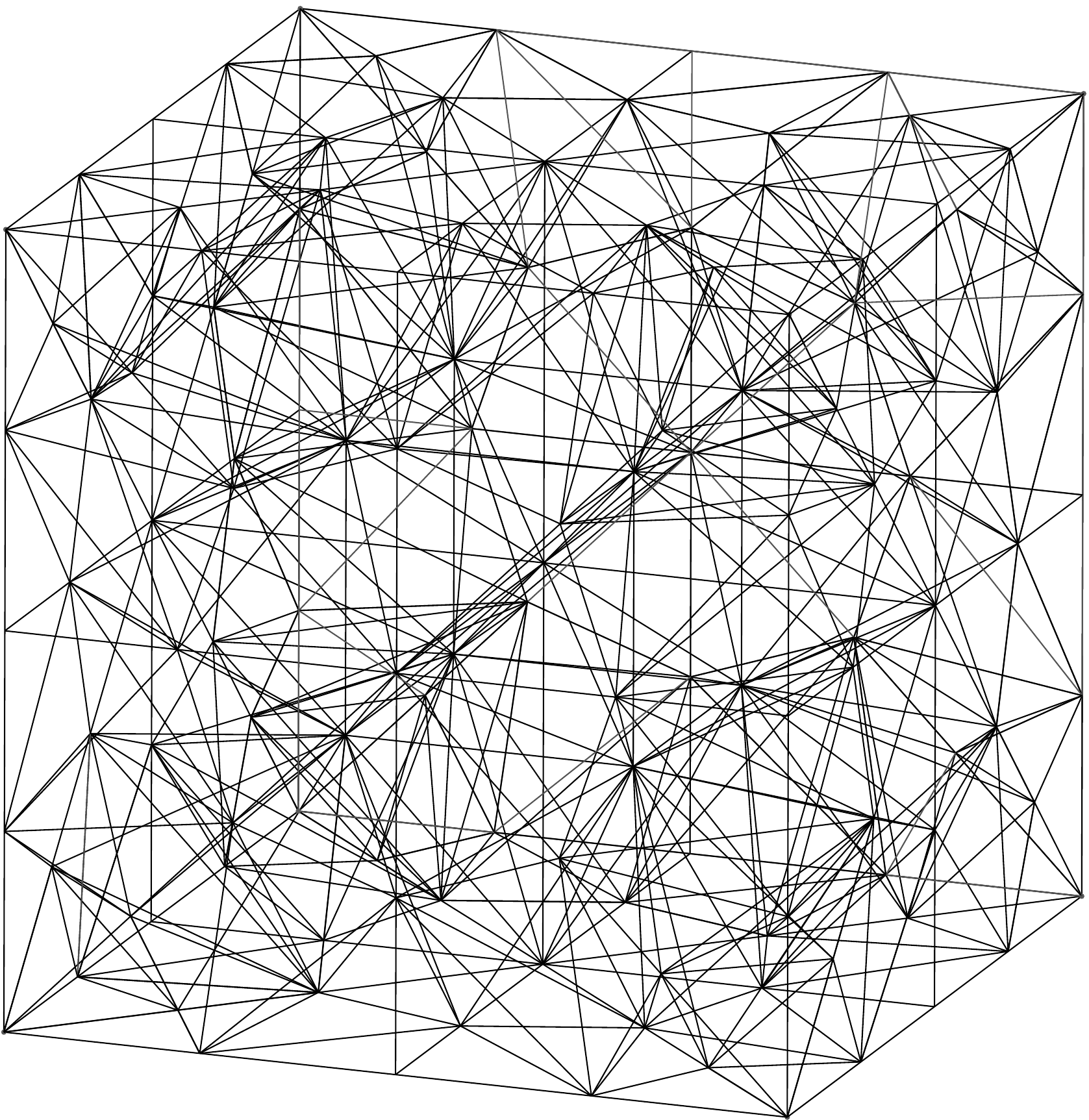}
  \caption{The triangular mesh (left), polygonal mesh (middle) and tetrahedral mesh (right).}
  \label{fig_mesh}
\end{figure}

\subsection{Study on convergence rate}
\label{subsec_convergence}
We first demonstrate the convergence behavior to examine the
theoretical predictions and show the efficiency of the proposed
method. 

\paragraph{\textbf{Example 1.}}
In the first example, we consider an elliptic problem in the squared
domain $\Omega = (-1, 1)^2$. The coefficient matrix $A$ is taken as
the identical matrix $I$, and the exact solution $u(x, y)$ is given by
\begin{displaymath}
  u(x,y) = \sin(2\pi(x+y)) \sin(2\pi y) + x^2 y.
\end{displaymath}
We solve this problem with the reconstructed space $U_h^m$, and
the numerical results are displayed in Tab.~\ref{tab_ex1err}. 
Clearly, the
numerical errors under both the energy norm and the $L^2$ norm
approach zero at the optimal convergence rates, which confirm the
theoretical results. 

\substitute{}{According to \cite{Hughes2000comparison}, the number of
degrees of freedom to the specific discrete system can serve as a
proper indicator for the scheme's efficiency. This indicator mainly
reveals the efficiency of the finite element approximation to Sobolev
spaces.} For the proposed scheme, the number of degrees
of freedom is always the number of elements in the partition.  
Let $V_h^m := \{ v_h \in L^2(\Omega) \ | \ v_h|_K \in
\mb{P}_m(K), \ \forall K \in \MTh \}$, which is usually the
approximation space in the standard DG method.
We also solve this problem by the space $V_h^m$.  The $L^2$ errors for
both methods against the number of degrees of freedom are plotted in
Fig.~\ref{fig_ex12_L2err}. 
It can be observed that fewer degrees of freedom are needed by the
reconstructed space to achieve a comparable $L^2$ error.
Tab.~\ref{tab_ex1_dofs} lists the ratio of the number of degrees of
freedom by two methods when the same $L^2$ errors are achieved. The
saving of degrees of freedom is more remarkable for $U_h^m$ when using
the high-order approximation.

\begin{table}[htbp]
  \centering
  \renewcommand{\arraystretch}{1.25}
  \scalebox{0.9}{
  \begin{tabular}{p{0.5cm}|p{2cm}|p{1.5cm}|p{1.5cm}|p{1.5cm}|p{1.5cm}|p{1.5cm}|p{1.5cm}}
    \hline\hline
    $m$   & $h$ & $1/10$ & $1/20$ & $1/40$ & $1/80$ & $1/160$ & order \\ \hline
    \multirow{2}{*}{1} & $\DGtenorm{u-u_h}$ 
    & 3.98e-0 & 1.93e-0 & 9.30e-1 & 4.55e-1 & 2.25e-1 & 1.03 \\  \cline{2-8}
    & $\|u-u_h\|_{L^2(\Omega)}$ 
    & 1.53e-1 & 3.50e-2 & 9.44e-3 & 2.42e-3 & 6.09e-4 & 2.00 \\ \hline
    \multirow{2}{*}{2} & $\DGtenorm{u-u_h}$ 
    & 1.78e-0 & 4.34e-1 & 1.02e-1 & 2.44e-2 & 5.95e-3 & 2.05 \\ \cline{2-8} 
    & $\|u-u_h\|_{L^2(\Omega)}$ 
    & 4.69e-2 & 5.41e-3 & 5.93e-4 & 6.81e-5 & 8.17e-6 & 3.06 \\ \hline
    \multirow{2}{*}{3} & $\DGtenorm{u-u_h}$
    & 7.36e-1 & 8.70e-2 & 1.03e-2 & 1.18e-3 & 1.38e-4 & 3.09 \\ \cline{2-8}
    &  $\|u-u_h\|_{L^2(\Omega)}$
    & 1.43e-2 & 7.04e-4 & 3.96e-5 & 2.40e-6 & 1.49e-7 & 4.01 \\ \hline
    \multirow{2}{*}{4} & $\DGtenorm{u-u_h}$
    & 3.34e-1 & 2.02e-2 & 1.32e-3 & 8.03e-5 & 4.80e-6 & 4.02 \\ \cline{2-8}
    &  $\|u-u_h\|_{L^2(\Omega)}$
    & 8.65e-3 & 2.38e-4 & 6.85e-6 & 1.96e-7 & 5.75e-9 & 5.08 \\ 
    \hline\hline
  \end{tabular}}
  \caption{The convergence histories for Example 1 with the
  reconstructed space $U_h^m$.}
  \label{tab_ex1err}
\end{table}

\begin{table}
  \centering
  \renewcommand{\arraystretch}{1.3}
  \scalebox{0.9}{
  \begin{tabular}{p{2cm}|p{1cm}|p{1cm}|p{1cm}|p{1cm}}
    \hline\hline
    $m$ & 1 & 2 & 3 & 4 \\ 
    \hline
    RDA/DGM & 72\% & 53\% & 42\% & 34\% \\ 
    \hline\hline
  \end{tabular}}
  \caption{The ratio of the number of degrees of freedom
  used in $U_h^m$ and $V_h^m$ when achieving a comparable $L^2$ error
  in Example 1.}
  \label{tab_ex1_dofs}
\end{table}

\paragraph{\textbf{Example 2.}}
Next, we consider the elliptic problem in three dimensions defined on
the cubic domain $\Omega = (0, 1)^3$. We choose the smooth function
\begin{displaymath}
  u(x,y, z) = \sin(2\pi(x + y + z)),
\end{displaymath}
as the exact solution. The coefficient matrix $A$ is taken as the
identical matrix $I$.
The convergence histories with the reconstructed space $U_h^m$ are
gathered in Tab.~\ref{tab_ex2err}, which clearly illustrate the
predictions in Theorem \ref{th_ellerror}.
Then, we solve this problem by the space $V_h^m$, and the numerical
results are reported in Fig.~\ref{fig_ex12_L2err} and
Tab.~\ref{tab_ex2_dofs}. 
One can observe that the reconstructed space $U_h^m$ still has a
better performance on the efficiency of \substitute{approximation
efficiency}{the finite element
approximation}, and the advantage
becomes prominent with the increasing of $m$ in three
dimensions.

\begin{table}[htbp]
  \centering
  \renewcommand{\arraystretch}{1.3}
  \scalebox{0.9}{
  \begin{tabular}{p{0.5cm}|p{2cm}|p{1.5cm}|p{1.5cm}|p{1.5cm}|p{1.5cm}|p{1.5cm}}
    \hline\hline
    $m$ &  $h$ & $1/4$ & $1/8$ & $1/16$ & $1/32$ & order \\ \hline
    \multirow{2}{*}{1} & $\DGtenorm{u-u_h}$
    & 1.23e-0 & 5.95e-1 & 2.89e-1 & 1.43e-1 & 1.03 \\ \cline{2-7}
    &  $\|u-u_h\|_{L^2(\Omega)}$
    & 5.46e-2 & 1.41e-2 & 3.70e-3 & 1.00e-3 & 1.92 \\ \hline
    \multirow{2}{*}{2} & $\DGtenorm{u-u_h}$  
    & 3.93e-1 & 9.52e-2 & 2.32e-2 & 5.68e-3 & 2.03 \\ \cline{2-7}
    &  $\|u-u_h\|_{L^2(\Omega)}$ 
    & 1.51e-2 & 1.99e-3 & 2.43e-4 & 3.09e-5 & 2.98 \\ \hline
    \multirow{2}{*}{3} & $\DGtenorm{u-u_h}$ 
    & 1.81e-1 & 1.74e-2 & 1.93e-2 & 2.30e-4 & 3.20 \\ \cline{2-7}
    &  $\|u-u_h\|_{L^2(\Omega)}$
    & 7.78e-3 & 3.62e-4 & 2.04e-5 & 1.25e-6 & 4.20 \\ 
    \hline\hline
  \end{tabular}}
  \caption{The convergence histories for Example 2 with the reconstructed space $U_h^m$.}
  \label{tab_ex2err}
\end{table}

\begin{figure}[htbp]
  \centering
  \includegraphics[width=0.45\textwidth]{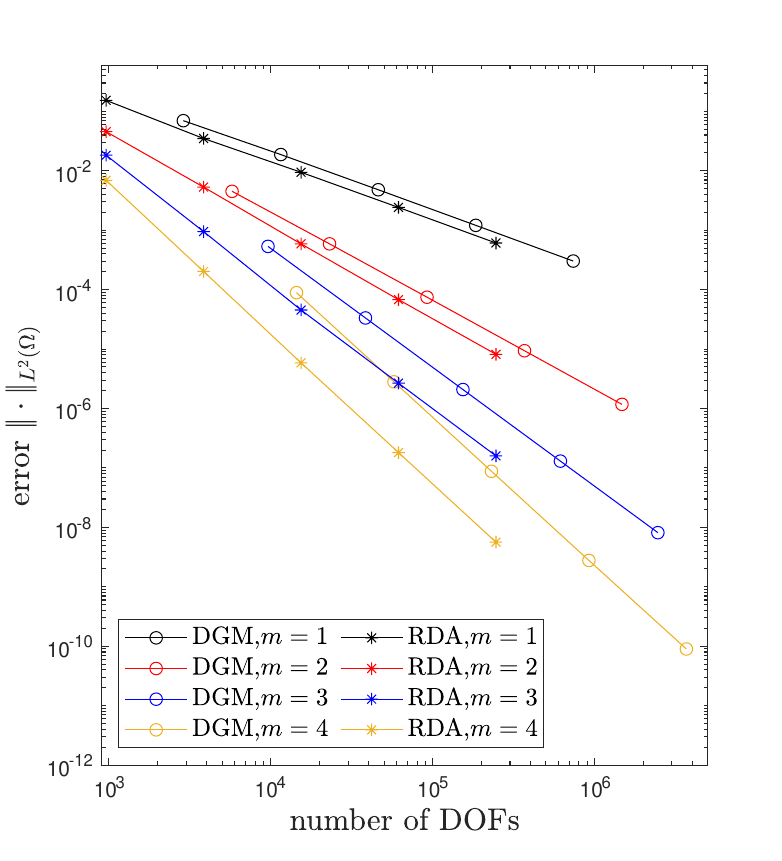}
  \includegraphics[width=0.45\textwidth]{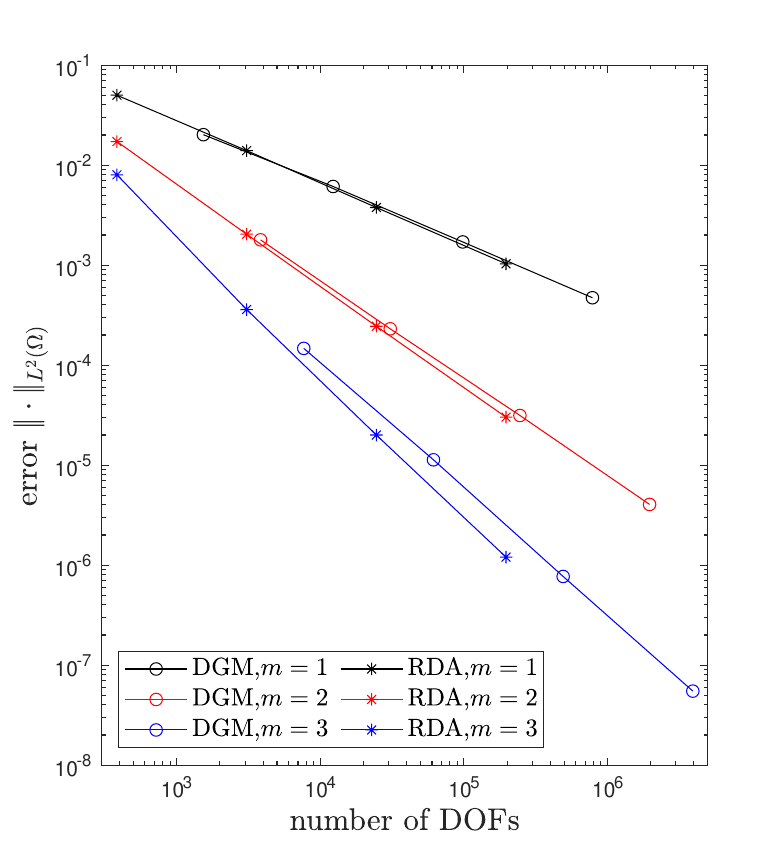}
  \caption{The comparison of the $L^2$ error for the spaces $U_h^m$ and $V_h^m$ in Example 1 (left) / Example 2
  (right).}
  \label{fig_ex12_L2err}
\end{figure}

\begin{table}
  \centering
  \renewcommand{\arraystretch}{1.3}
  \scalebox{0.9}{
  \begin{tabular}{p{2cm}|p{1cm}|p{1cm}|p{1cm}}
    \hline\hline
    $m$ & 1 & 2 & 3 \\
    \hline
    RDA/DGM & 57\% & 46\% & 34\% \\
    \hline\hline
  \end{tabular}}
  \caption{The ratio of the number of degrees of freedom used in
  $U_h^m$ and $V_h^m$ when achieving a comparable $L^2$ error in
  Example 2.}
  \label{tab_ex2_dofs}
\end{table}

\subsection{Study on preconditioned system}
\label{subsec_precondition} 
In this subsection, we illustrate the efficiency on solving the
preconditioned system $A_0^{-1} A_{m, \theta} \bmr{x} = \bmr{b}$ in
two and three dimensions.  \substitute{For both symmetric and
nonsymmetric schemes, the linear system is solved by the
preconditioned GMRES iterative method, and the preconditioner is
chosen as the multigrid method for $A_0^{-1}$ given in Algorithm
\ref{alg_multigrid}. The number of the presmoothing steps, the error
correction steps and the postsmoothing steps in Algorithm
\ref{alg_multigrid} are taken to be $10, 1, 10$.}{ For the symmetric
and nonsymmetric schemes, the linear systems are solved by the
preconditioned CG/GMRES methods, respectively. 
In Algorithm \ref{alg_multigrid}, the number of the
presmoothing and postsmoothing steps are fixed as $4$. We also use the
BoomerAMG method and the direct LU method to approximate
$A_0^{-1}$ to precondition the linear systems.} 
The iteration stops when the relative error 
\begin{displaymath}
  \frac{\| \bmr{b} - A_{m,\theta} \bmr{x}_k \|_{l^2}}{\| \bmr{b} 
  \|_{l^2}}.
\end{displaymath}
at the stage $k$ is smaller than the tolerance $\varepsilon = 10^{-8}$.

\paragraph{\textbf{Example 3.}}
In this example, we examine the condition numbers of the preconditioned
system $A_0^{-1} A_{m, \theta}$ for both $\theta = \pm 1$.  We
assemble the linear system on the triangular meshes with $h = 1/10,
1/20, 1/40$.  The condition numbers for different $m$ are
gathered in Tab.~\ref{tab_cond2d}.  It can be observed that the condition
numbers for both symmetric/nonsymmetric systems are nearly constants
as the mesh size $h$ tends to zero.  This numerical observation fairly
matches the estimate in Theorem \ref{th_A0Amcond}.

\begin{table}[htbp]
  \centering
  \renewcommand{\arraystretch}{1.2}
  \begin{tabular}{p{1.3cm}|p{1.6cm}|p{1.5cm}|p{1.5cm} || p{1.5cm}|
    p{1.5cm}|p{1.5cm}}
    \hline\hline
    \multirow{2}{*}{$h$} &  \multicolumn{3}{c||}{$\kappa(A_0^{-1} A_{m, -1})$} 
    &   \multicolumn{3}{c}{$\kappa(A_0^{-1} A_{m, 1})$} \\
    \cline{2-7}
    & $m=1$ & $m=2$ & $m=3$ & $m=1$ & $m=2$ & $m=3$ \\
    \hline 
    $1/10$ & 10.73 & 14.81 & 29.28 & 11.48 & 21.61 & 40.90 \\ 
    \hline
    $1/20$ & 11.14 & 17.62 & 27.99 & 12.23 & 23.28 & 39.29 \\ 
    \hline
    $1/40$ & 12.78 & 18.81 & 29.73 & 13.70 & 24.89 & 41.61 \\ 
    \hline\hline
  \end{tabular}
  \caption{The condition numbers of the preconditioned systems in two dimensions.}
  \label{tab_cond2d}
\end{table}

\paragraph{\textbf{Example 4.}}  
In this example, we solve an elliptic problem on the domain $\Omega =
(-1, 1)^2$. We adopt a family of triangular meshes with the mesh size
$h = 1/10, 1/20, 1/40, 1/80, 1/160$. 
The exact solution is chosen to be the same as in Example 1.
\substitute{}{Tab.~\ref{tab_ex4_steps} lists the 
convergence steps for CG/GMRES solvers. 
Besides preconditioning with $A_0^{-1}$, we try to use the
BoomerAMG algorithm to the matrix $A_{m, \theta}$ as the
preconditioner, and we also directly apply the standard CG/GMRES
methods in solving linear systems.
It can be observed that the methods preconditioning with $A_0^{-1}$
have a much faster convergence speed than other methods.}
For all accuracy $1 \leq m \leq 4$, the convergence steps are
numerically detected to be independent of $h$, which
illustrates the results in Theorem \ref{th_A0Amcond}. 
Fig.~\ref{fig_ex4_relative_error} displays the convergence histories
of CG/GMRES solvers. 
The relative errors in iterations of the system $A_0^{-1} A_{m,
\theta}$ decrease sharply, compared
to the standard CG/GMRES methods. 
Tab.~\ref{tab_ex4_times} gives the CPU times in solving the linear
system by different methods. 
It is evident that the CG/GMRES methods, when employing the Algorithm
\ref{alg_multigrid} for $A_0^{-1}$ as a preconditioner, is faster than
other methods.

As we stated in Remark \ref{re_Lambda}, the element patch will bring
more computational cost in assembling the stiffness matrix and
increase the width of the banded structure.
Tab.~\ref{tab_ex4_alltime} lists the CPU times in different steps for
both our symmetric scheme and the standard symmetric DG scheme.
It can be observed that for our method, solving local least squares
problems per element is very cheap, and as expected, our method
requires more CPU time in the step of assembling the stiffness matrix. 
For our method, the most time-consuming step is assembling the
stiffness matrix, and solving the final linear system is much faster
than this step. Generally, the step of assembling the stiffness matrix
can be easily accelerated by a parallel implementation typically
exploiting the multithreading technique.
For the DG method solving this example, we also try to apply the
preconditioned CG method using BoomerAMG algorithm as the
preconditioner. 
\revise{ The CPU time is essentially dominated by solving the linear
system, which is different from the reconstructed method. }{
We observe that the step of solving the linear system is more
time-consuming compared with assembling the stiffness matrix, which
appears to be different from the reconstructed method.
}

\begin{table}[htbp]
  \centering
  \renewcommand{\arraystretch}{1.15}
  \scalebox{0.80}{
  \revise{}{
  \begin{tabular}{p{0.25cm}|p{0.5cm}|p{2.5cm}|p{0.85cm}|p{0.85cm}|p{0.85cm}|p{0.85cm}
    |p{0.85cm}||p{0.85cm}|p{0.85cm}|p{0.85cm}|p{0.85cm}|p{0.85cm}}
    \hline\hline
    \multirow{2}{*}{$m$} &
    \multicolumn{2}{c|}{\diagbox[width=3.85cm]{Preconditioner}{$1/h$}}
    & 10 & 20 & 40 & 80 & 160 & 10 & 20 & 40 & 80 & 160 \\
    \cline{2-13}
    & \multicolumn{2}{c|}{$\theta$} & \multicolumn{5}{c||}{-1:
    symmetric}  & \multicolumn{5}{c}{1: nonsymmetric} \\
    \hline
    \multirow{5}{*}{1} & \multirow{3}{*}{$A_0^{-1}$} & GMG & 
    16 & 17 & 18 & 18 & 18 & 
    23 & 25 & 25 & 26 & 26 \\
    \cline{3-13}
    & & BoomerAMG & 
    16 & 17 & 18 & 19 & 20 &
    23 & 24 & 26 & 27 & 28 \\ 
    \cline{3-13}        
    & & DirectSolver & 
    16 & 17 & 18 & 18 & 18 &
    23 & 24 & 25 & 25 & 26 \\ \cline{2-13}        
    & \multicolumn{2}{c|}{BoomerAMG} &
    15 & 16 & 20 & 26 & 35 & 
    16 & 20 & 25 & 32 & 47 \\ \cline{2-13}
    & \multicolumn{2}{c|}{Identity} & 
    109 & 251 & 507 & 1019 & 1994 & 
    80 & 269 & 827 & 2690 & - \\ \hline
    \multirow{5}{*}{2} & \multirow{3}{*}{$A_0^{-1}$} & GMG & 
    20 & 22 & 23 & 23 & 23 & 
    32 & 32 & 33 & 34 & 35 \\ \cline{3-13}
    \cline{3-13}
    & & BoomerAMG & 
    20 & 23 & 24 & 24 & 24 &
    33 & 33 & 35 & 35 & 36 \\ \cline{3-13}        
    & & DirectSolver & 
    20 & 21 & 23 & 23 & 23 &
    32 & 32 & 32 & 33 & 33 \\ \cline{2-13}        
    & \multicolumn{2}{c|}{BoomerAMG} & 
    20 & 25 & 32 & 41 & 63 &
    24 & 26 & 31 & 39 & 56 \\ \cline{2-13}
    & \multicolumn{2}{c|}{Identity} & 
    274 & 548 & 1067 & 2071 & - &
    173 & 432 & 994 & 2722 & - \\ \hline
    \multirow{5}{*}{3} & \multirow{3}{*}{$A_0^{-1}$} & GMG & 
    44 & 45 & 45 & 46 & 46 & 
    49 & 51 & 52 & 53 & 54 \\
    \cline{3-13}
    & & BoomerAMG & 
    43 & 45 & 45 & 46 & 46 &
    48 & 51 & 51 & 51 & 52 \\ \cline{3-13}
    & & DirectSolver &
    41 & 43 & 44 & 44 & 44 & 
    48 & 50 & 49 & 49 & 50 \\ \cline{2-13}        
    & \multicolumn{2}{c|}{BoomerAMG} &
    32 & 35 & 42 & 53 & 70 & 
    35 & 38 & 41 & 49 & 64 \\ \cline{2-13}
    & \multicolumn{2}{c|}{Identity} & 
    411 & 806 & 1598 & - & - &
    238 & 899 & 2345 & - & - \\ \hline
    \multirow{5}{*}{4} & \multirow{3}{*}{$A_0^{-1}$} & GMG & 
    68 & 72 & 73 & 74 & 75 &
    60 & 62 & 69 & 74 & 77 \\ \cline{3-13}
    & & BoomerAMG & 
    63 & 68 & 70 & 71 & 71 & 
    59 & 61 & 66 & 71 & 72 \\ \cline{3-13}
    & & DirectSolver & 
    61 & 66 & 68 & 69 & 68 & 
    59 & 61 & 66 & 71 & 73 \\ \cline{2-13}        
    & \multicolumn{2}{c|}{BoomerAMG} & 
    60 & 63 & 76 & 86 & 102 & 
    62 & 67 & 76 & 91 & 105 \\ \cline{2-13}
    & \multicolumn{2}{c|}{Identity} & 
    588 & 1264 & 2528 & - & - &
    290 & 1390 & - & - & - \\
    \hline\hline
  \end{tabular}
  }}
  \caption{The convergence steps for symmetric and nonsymmetric
  methods in Example 4.}
  \label{tab_ex4_steps}
\end{table}

\begin{table}[htbp]
  \centering
  \renewcommand{\arraystretch}{1.15}
  \scalebox{0.75}{
  \revise{}{
  \begin{tabular}{p{0.2cm}|p{0.5cm}|p{2.5cm}|p{1.05cm}|p{1.05cm}|p{1.05cm}|p{1.05cm} | p{1.05cm}||
    p{1.05cm}|p{1.05cm}|p{1.05cm}|p{1.05cm}|p{1.05cm}}
    \hline\hline
    \multirow{2}{*}{$m$} & \multicolumn{2}{c|}{\diagbox[width=3.5cm]{\small Preconditioner}{$1/h$}} & 10 & 20 & 40 & 80 & 
    160 & 10 & 20 & 40 & 80 & 160 \\ \cline{2-13}
    & \multicolumn{2}{c|}{$\theta$} & \multicolumn{5}{c||}{-1:symmetric } 
    & \multicolumn{5}{c}{1:nonsymmetric} \\ \hline
    \multirow{5}{*}{1} & \multirow{3}{*}{$A_0^{-1}$} & GMG & 
     0.002 & 0.010 & 0.036 & 0.160 & 0.744 
    & 0.002 & 0.011 & 0.044 & 0.200 & 0.982 \\ 
    \cline{3-13}
    & & BoomerAMG & 
    0.005 & 0.017 & 0.063 & 0.282 & 1.229 &
    0.007 & 0.022 & 0.080 & 0.345 & 1.559
    \\ \cline{3-13}        
    & & DirectSolver & 
    0.002 & 0.008 & 0.036 & 0.255 & 1.396 &
    0.002 & 0.008 & 0.042 & 0.301 & 1.736 
    \\\cline{2-13}        
    & \multicolumn{2}{c|}{BoomerAMG} & 
    0.010 & 0.036 & 0.156 & 0.962 & 5.616 & 
    0.018 & 0.063 & 0.268 & 1.195 & 7.347 \\ \cline{2-13}        
    & \multicolumn{2}{c|}{Identity} &
    0.015 & 0.096 & 0.753 & 6.025 & 58.90
    & 0.009 & 0.065 & 0.614 & 8.416 & - \\\hline
    \multirow{5}{*}{2} & \multirow{3}{*}{$A_0^{-1}$} & GMG &
    0.004 & 0.013 & 0.057 & 0.256 & 1.112 &
    0.004 & 0.018 & 0.075 & 0.352 & 1.497 \\ \cline{3-13}
    & & BoomerAMG & 
    0.008 & 0.022 & 0.085 & 0.376 & 1.603 &
    0.012 & 0.028 & 0.126 & 0.533 & 2.219 \\\cline{3-13}        
    & & DirectSolver & 
    0.004 & 0.012 & 0.056 & 0.423 & 1.923 & 
    0.004 & 0.015 & 0.075 & 0.493 & 2.765 \\\cline{2-13}        
    & \multicolumn{2}{c|}{BoomerAMG} & 
    0.016 & 0.069 & 0.336 & 1.851 & 12.22 &
    0.016 & 0.065 & 0.300 & 1.734 & 10.77 \\ \cline{2-13}        
    & \multicolumn{2}{c|}{Identity} & 
    0.026 & 0.187 & 1.530 & 13.60 & - &
    0.013 & 0.142 & 1.341 & 20.23 & - \\\hline
    \multirow{5}{*}{3} & \multirow{3}{*}{$A_0^{-1}$} & GMG &
    0.020 & 0.036 & 0.148 & 0.662 & 2.833 &
    0.007 & 0.032 & 0.141 & 0.690 & 2.955 \\ \cline{3-13}
    & & BoomerAMG &
    0.013 & 0.049 & 0.187 & 0.820 & 3.512 &
    0.013 & 0.059 & 0.203 & 0.854 & 3.650 \\ \cline{3-13}
    & & DirectSolver & 
    0.009 & 0.032 & 0.150 & 0.862 & 4.479 & 
    0.006 & 0.027 & 0.132 & 1.093 & 4.625 \\\cline{2-13}        
    & \multicolumn{2}{c|}{BoomerAMG} & 
    0.028 & 0.125 & 0.610 & 3.669 & 18.85 & 
    0.030 & 0.130 & 0.582 & 3.933 & 17.05 \\ \cline{2-13}        
    & \multicolumn{2}{c|}{Identity} & 
    0.058 & 0.427 & 3.529 & - & -  &
    0.022 & 0.352 & 3.823 & - & - \\\hline
    \multirow{5}{*}{4} & \multirow{3}{*}{$A_0^{-1}$} & GMG &
    0.019 & 0.077 & 0.324 & 1.546 & 6.388 & 
    0.012 & 0.049 & 0.226 & 1.130 & 5.300  \\ \cline{3-13}
    & & BoomerAMG & 
    0.024 & 0.089 & 0.370 & 1.567 & 7.229 &
    0.018 & 0.066 & 0.304 & 1.649 & 7.544 \\ \cline{3-13}
    & & DirectSolver & 
    0.016 & 0.068 & 0.327 & 1.856 & 8.563 & 
    0.016 & 0.043 & 0.237 & 1.460 & 7.847 \\\cline{2-13}        
    & \multicolumn{2}{c|}{BoomerAMG} & 
    0.060 & 0.281 & 1.603 & 8.587 & 42.99 &
    0.063 & 0.286 & 1.390 & 8.981 & 42.58 \\ \cline{2-13}        
    & \multicolumn{2}{c|}{Identity} & 
    0.120 & 1.027 & 8.490 & - & - &
    0.038 & 0.740 & - & - & - \\
    \hline\hline
  \end{tabular}}}
  \caption{The CPU times of solving linear systems for symmetric and
  nonsymmetric methods in Example 4.}
  \label{tab_ex4_times}
\end{table}
\begin{table}
  \centering
  \renewcommand{\arraystretch}{1.25}
  \scalebox{0.75}{
  \revise{}{
  \begin{tabular}{p{0.3cm}| p{1.0cm} | p{7.0cm} |  p{1cm} | p{1cm} |
    p{1cm}  }
    \hline\hline
    $m$ & & $h$ & $1/20$ & $1/40$ &  $1/80$  \\
    \hline
    \multirow{5}{*}{$1$} & \multirow{3}{*}{RDA} & 
    solve least squares problems & 
    0.015 &  0.061 & 0.251 \\
    \cline{3-6}
    & & assemble the stiffness matrix & 
    0.155  & 0.636 & 2.626 \\ 
    \cline{3-6}
    & & solve the linear system by CG + $A_0^{-1}$  &
    0.010 & 0.036 & 0.160 \\ 
    \cline{2-6}
    & \multirow{2}{*}{DGM} &   assemble
    the stiffness matrix &
    0.052  & 0.206 & 0.826 \\ 
    \cline{3-6}
    & & solve the linear system by CG + AMG & 
    0.076 & 0.313 & 1.662 \\ 
    \hline
    \multirow{5}{*}{$2$} & \multirow{3}{*}{RDA} & 
    solve least squares problems &
    0.047 &  0.181 & 0.739 \\
    \cline{3-6}
    & & assemble the stiffness matrix &
    0.578  & 2.332 & 9.378 \\ 
    \cline{3-6}
    & & solve the linear system by CG + $A_0^{-1}$ &
    0.013 & 0.057 & 0.256 \\ 
    \cline{2-6}
    & \multirow{2}{*}{DGM} &   assemble
    the stiffness matrix &
    0.287  & 1.150 & 4.638 \\ 
    \cline{3-6}
    & & solve the linear system by CG + AMG &
    0.272 & 1.883 & 11.79 \\ 
    \hline
    \multirow{5}{*}{$3$} & \multirow{3}{*}{RDA} & 
    solve least squares problems & 
    0.113 &  0.463 & 1.861 \\
    \cline{3-6}
    & & assemble the stiffness matrix & 
    1.092  & 4.369 & 17.37 \\ 
    \cline{3-6}
    & & solve the linear system by CG + $A_0^{-1}$ & 
    0.036 & 0.148 & 0.662 \\ 
    \cline{2-6}
    & \multirow{2}{*}{DGM} & assemble the stiffness matrix & 
    0.568  & 2.280 & 9.576 \\ 
    \cline{3-6}
    & & solve the linear system by CG + AMG &
    0.593 & 4.372 & 29.06 \\ 
    \hline
    \multirow{5}{*}{$4$} & \multirow{3}{*}{RDA} & 
    solve least squares problems & 
    0.265 &  1.066 & 4.260 \\
    \cline{3-6}
    & & assemble the stiffness matrix & 
    2.392  & 9.301 & 37.23 \\ 
    \cline{3-6}
    & & solve the linear system by CG + $A_0^{-1}$& 
    0.077 & 0.324 & 1.546 \\ 
    \cline{2-6}
    & \multirow{2}{*}{DGM} &   
    assemble the stiffness matrix & 
    1.382  & 5.391 & 20.25 \\ 
    \cline{3-6}
    & & solve the linear system  by CG + AMG&
    1.723 & 12.29 & 67.29 \\ 
    \hline\hline
  \end{tabular}}}
  \caption{The CPU times for all steps in Example 4.}
  \label{tab_ex4_alltime}
\end{table}

\begin{figure}[htbp]
  \centering
  \includegraphics[width=0.35\textwidth]{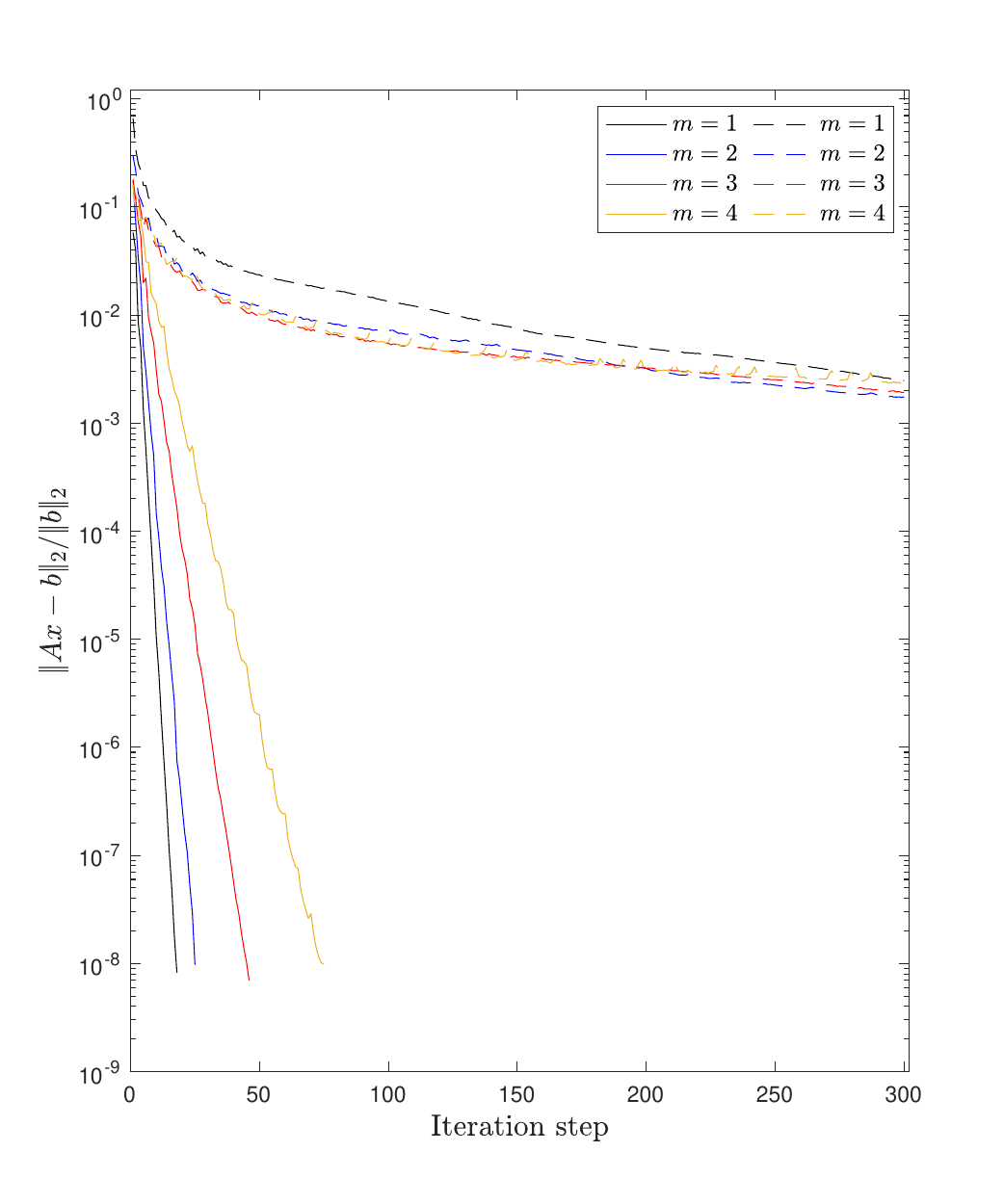}
  \hspace{20pt}
  \includegraphics[width=0.35\textwidth]{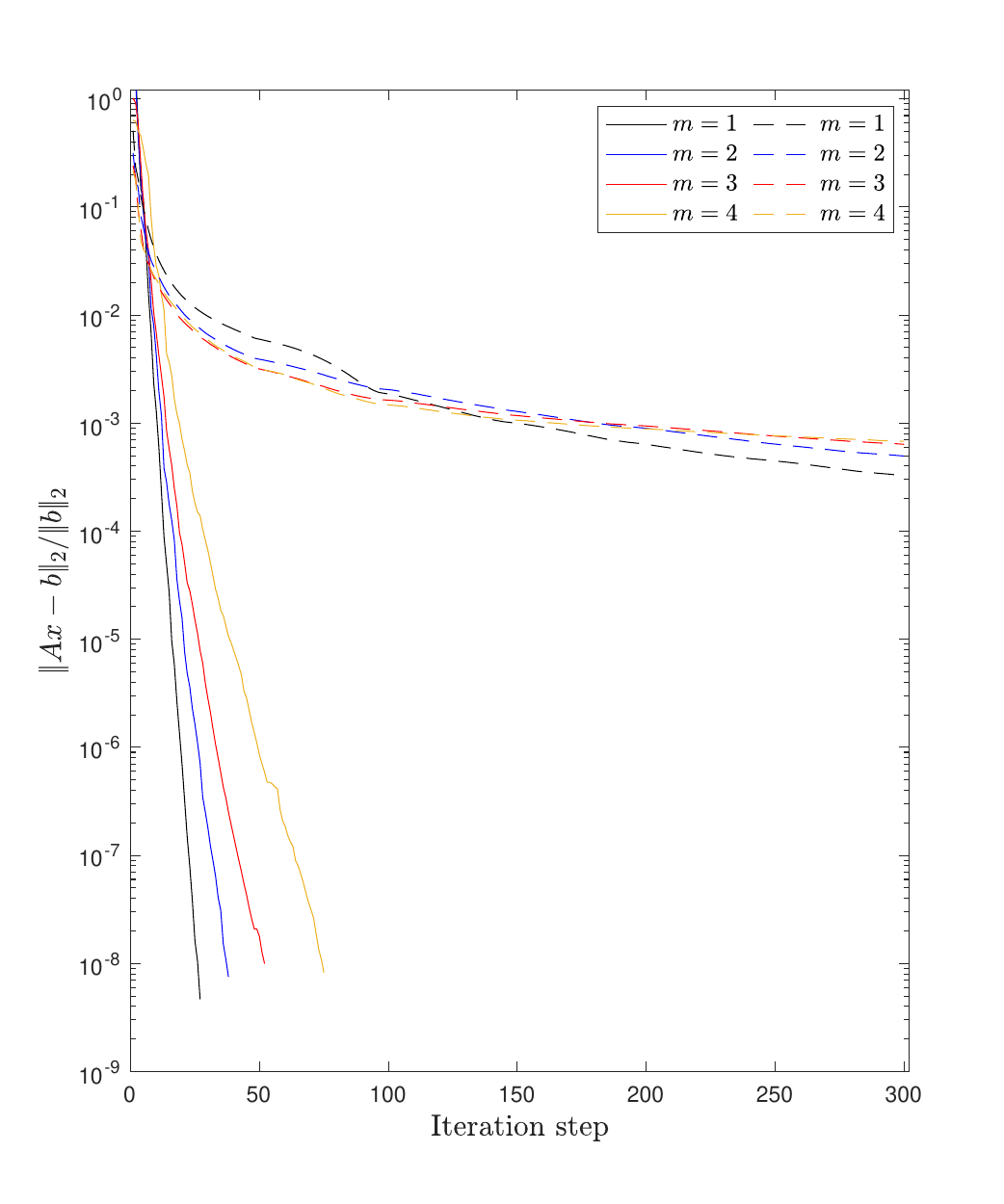}
  \caption{The convergence histories of CG (left) /GMRES (right) solvers on the mesh $h =
  1/160$ (solid line: CG/GMRES solver with preconditioner $A_0^{-1}$,
  dashed line: standard CG/GMRES solver).}
  \label{fig_ex4_relative_error}
\end{figure}

\paragraph{\textbf{Example 5.}}
In this example, we numerically solve the problem defined on $\Omega =
(-1, 1)^2$ with a series of polygonal meshes. The polygonal meshes are
generated by the package PolyMesher \cite{talischi2012polymesher},
which contain different types of polygon elements. The exact solution
is
\begin{displaymath}
  u(x,y) = e^{x^2 + y^2} \sin(xy).
\end{displaymath}
Because the polygonal meshes are not nested, only the algebraic
multigrid method and the direct method are used to approximate
$A_0^{-1}$.
The numerical results
are shown in Tab.~\ref{tab_ex5_steps} and Tab.~\ref{tab_ex5_times}.
Clearly, the results demonstrate that the proposed method can work
very well for elements of complex geometries and has a great
efficiency on solving the linear system.

\begin{table}[htbp]
  \centering
  \renewcommand{\arraystretch}{1.15}
  \scalebox{0.80}{
  \revise{}{
  \begin{tabular}{p{0.25cm}|p{0.5cm}|p{2.5cm}|p{0.85cm}|p{0.85cm}|p{0.85cm}|p{0.85cm}
    |p{0.85cm}||p{0.85cm}|p{0.85cm}|p{0.85cm}|p{0.85cm}|p{0.85cm}}
    \hline\hline
    \multirow{2}{*}{$m$} & \multicolumn{2}{c|}{\diagbox[width=3.85cm]{Preconditioner}{$1/h$}} & 10 & 20 & 40 & 80 & 160 
    & 10 & 20 & 40 & 80 & 160 \\ \cline{2-13}
    & \multicolumn{2}{c|}{$\theta$} & \multicolumn{5}{c||}{-1: symmetric}  & \multicolumn{5}{c}{1: nonsymmetric} \\ \hline
    \multirow{4}{*}{1} & \multirow{2}{*}{$A_0^{-1}$} & BoomerAMG & 
    24 & 27 & 29 & 30 & 31 &
    19 & 20 & 21 & 21 & 21 \\
    \cline{3-13}        
    & & DirectSolver & 
    23 & 26 & 28 & 28 & 29 & 
    19 & 20 & 21 & 21 & 21 \\ \cline{2-13}        
    & \multicolumn{2}{c|}{BoomerAMG} & 
    13 & 15 & 25 & 32 & 52 & 
    13 & 15 & 23 & 28 & 39 \\ \cline{2-13}        
    & \multicolumn{2}{c|}{Identity} &
    130 & 325 & 623 & 1230 & 2368 & 
    134 & 302 & 592 & 1406 & - \\ \hline
    \multirow{4}{*}{2} & \multirow{2}{*}{$A_0^{-1}$} & BoomerAMG & 
    36 & 41 & 44 & 46 & 47 &
    25 & 29 & 31 & 31 & 31 \\
    \cline{3-13}        
    & & DirectSolver &
    36 & 40 & 43 & 44 & 44 & 
    26 & 29 & 31 & 31 & 31 \\ \cline{2-13}        
    & \multicolumn{2}{c|}{BoomerAMG} & 
    17 & 21 & 28 & 52 & 82 & 
    17 & 21 & 27 & 43 & 66 \\ \cline{2-13}        
    & \multicolumn{2}{c|}{Identity} & 
    150 & 415 & 800  & 1608 & -  & 
    148  & 446 & 825 & 1880 & - \\ \hline
    \multirow{4}{*}{3} & \multirow{2}{*}{$A_0^{-1}$} & BoomerAMG & 
    57 & 64 & 69 & 74 & 75 &
    33 & 37 & 40 & 40 & 41 \\
    \cline{3-13}        
    & & DirectSolver & 
    56 & 63 & 67 & 71 & 71 & 
    33 & 37 & 41 & 40 & 41 \\ \cline{2-13}        
    & \multicolumn{2}{c|}{BoomerAMG} & 
    33 & 42 & 49 & 62 & 98 & 
    32 & 39 & 46 & 59 & 74 \\ \cline{2-13}        
    & \multicolumn{2}{c|}{Identity} & 
    228 & 614 & 1213 & 2509 & - & 
    180  & 543 & 1110 & 2793 & - \\ \hline
    \multirow{4}{*}{4} & \multirow{2}{*}{$A_0^{-1}$} & BoomerAMG &
    76 & 91 & 101 & 116 & 119 & 
    38 & 44 & 53 & 57 & 59 \\
    \cline{3-13}        
    & & DirectSolver & 
    75 & 89 & 99 & 111 & 113 & 
    39 & 44 & 53 & 57 & 59 \\ \cline{2-13}        
    & \multicolumn{2}{c|}{BoomerAMG} & 
    62 & 75 & 95 & 116 & 156 & 
    55 & 72 & 89 & 105 & 129 \\ \cline{2-13}        
    & \multicolumn{2}{c|}{Identity} & 
    308 & 977 & 2168 & - & - & 
    186 & 787 & 1702 & - & - \\
    \hline\hline
  \end{tabular}
  }}
  \caption{The convergence steps for symmetric and nonsymmetric methods in Example 5.}
  \label{tab_ex5_steps}
\end{table}

\begin{table}[htbp]
  \centering
  \renewcommand{\arraystretch}{1.15}
  \scalebox{0.75}{
  \revise{}{
  \begin{tabular}{p{0.2cm}|p{0.5cm}|p{2.5cm}|p{1.05cm}|p{1.05cm}|p{1.05cm}|p{1.05cm} | p{1.05cm}||
    p{1.05cm}|p{1.05cm}|p{1.05cm}|p{1.05cm}|p{1.05cm}}
    \hline\hline
    \multirow{2}{*}{$m$} & \multicolumn{2}{c|}{\diagbox[width=3.5cm]{\small Preconditioner}{$1/h$}} & 10 & 20 & 40 & 80 
    & 160 & 10 & 20 & 40 & 80 & 160 \\ \cline{2-13}
    & \multicolumn{2}{c|}{$\theta$} & \multicolumn{5}{c||}{-1:symmetric}  & \multicolumn{5}{c}{1:nonsymmetric} \\ \hline
    \multirow{4}{*}{1} & \multirow{2}{*}{$A_0^{-1}$} & BoomerAMG & 
    0.002 & 0.009 & 0.033 & 0.131 & 0.555 &
    0.002 & 0.006 & 0.022 & 0.099 & 0.466 \\ \cline{3-13}
    & & DirectSolver & 
    0.001 & 0.003 & 0.014 & 0.067 & 0.516 & 
    0.001 & 0.002 & 0.010 & 0.047 & 0.317 \\\cline{2-13}        
    & \multicolumn{2}{c|}{BoomerAMG} & 
    0.002 & 0.008 & 0.051 & 0.280 & 1.861 & 
    0.003 & 0.011 & 0.056 & 0.227 & 1.682 \\ \cline{2-13}        
    & \multicolumn{2}{c|}{Identity} & 
    0.003 & 0.023 & 0.170 & 1.401 & 13.01 &
    0.004 & 0.036 & 0.288 & 2.979 & - \\ \hline
    \multirow{4}{*}{2} & \multirow{2}{*}{$A_0^{-1}$} & BoomerAMG & 
    0.003 & 0.016 & 0.052 & 0.239 & 1.197 & 
    0.003 & 0.010 & 0.038 & 0.151 & 0.709 \\ \cline{3-13}        
    & & DirectSolver & 
    0.003 & 0.008 & 0.032 & 0.151 & 0.979 & 
    0.001 & 0.004 & 0.018 & 0.082 & 0.555 \\ \cline{2-13}        
    & \multicolumn{2}{c|}{BoomerAMG} & 
    0.003 & 0.016 & 0.087 & 0.602 & 5.321 &
    0.005 & 0.018 & 0.087 & 0.560 & 4.171 \\ \cline{2-13}        
    & \multicolumn{2}{c|}{Identity} & 
    0.004 & 0.049 & 0.396 & 3.388 & - & 
    0.005 & 0.061 & 0.502 & 5.734 & - \\ \hline        
    \multirow{4}{*}{3} & \multirow{2}{*}{$A_0^{-1}$} & BoomerAMG & 
    0.005 & 0.023 & 0.090 & 0.403 & 1.793 &
    0.003 & 0.012 & 0.045 & 0.187 & 0.825 \\ \cline{3-13}
    & & DirectSolver & 
    0.002 & 0.013 & 0.059 & 0.304 & 1.789 & 
    0.002 & 0.006 & 0.026 & 0.132 & 0.836 \\\cline{2-13}        
    & \multicolumn{2}{c|}{BoomerAMG} & 
    0.007 & 0.033 & 0.174 & 0.988 & 7.758 & 
    0.007 & 0.034 & 0.169 & 0.912 & 6.984 \\ \cline{2-13}        
    & \multicolumn{2}{c|}{Identity} & 
    0.009 & 0.096 & 0.796 & 6.940 & - & 
    0.007 & 0.087 & 0.769 & 8.809 & - \\\hline
    \multirow{4}{*}{4} & \multirow{2}{*}{$A_0^{-1}$} & 
    BoomerAMG & 
    0.007 & 0.035 & 0.155 & 0.779 & 3.911 &
    0.004 & 0.015 & 0.066 & 0.330 & 1.674 \\ \cline{3-13}
    & & DirectSolver & 
    0.005 & 0.024 & 0.129 & 0.769 & 4.888 & 
    0.002 & 0.010 & 0.046 & 0.273 & 1.774 \\\cline{2-13}        
    & \multicolumn{2}{c|}{BoomerAMG} & 
    0.015 & 0.084 & 0.428 & 2.651 & 19.66 &
    0.014 & 0.082 & 0.432 & 2.406 & 15.66 \\ \cline{2-13}        
    & \multicolumn{2}{c|}{Identity} & 
    0.014 & 0.226 & 2.149 & - & -  &
    0.008 & 0.149 & 1.490 & - & -  \\
    \hline\hline
  \end{tabular}}}
  \caption{The CPU times of solving linear system for symmetric and
  nonsymmetric methods in Example 5.}
  \label{tab_ex5_times}
\end{table}

\paragraph{\textbf{Example 6.}}
In this case, we test an elliptic problem on the domain $\Omega = (-1,
1)^2$ with the coefficient matrix 
\begin{displaymath}
  A = \left(
  \begin{array}{cc}
    3 & 0 \\
    0 & 0.1 \\
  \end{array}
  \right),
\end{displaymath}
and the exact solution is selected by
\begin{displaymath}
  u(x,y) = \sin(\frac{1}{3} x) + \cos(10 y).
\end{displaymath}
In this test, the penalty parameter in the symmetric scheme is taken
as $\mu = 6m^2 + 10$.
For such a problem, the convergence steps and CPU time of solving the
final linear systems are recorded in Tab.~\ref{tab_ex6_steps} and
Tab.~\ref{tab_ex6_times}, respectively. The preconditioned methods
still have a good numerical performance for both symmetric and
nonsymmetric interior penalty methods.

\begin{table}[htbp]
  \centering
  \renewcommand{\arraystretch}{1.15}
  \scalebox{0.8}{
  \revise{}{
  \begin{tabular}{p{0.25cm}|p{0.5cm}|p{2.5cm}|p{0.85cm}|p{0.85cm}|p{0.85cm}|p{0.85cm}
    |p{0.85cm}||p{0.85cm}|p{0.85cm}|p{0.85cm}|p{0.85cm}|p{0.85cm}}
    \hline\hline
    \multirow{2}{*}{$m$} & \multicolumn{2}{c|}{\diagbox[width=3.85cm]{Preconditioner}{$1/h$}} & 10 & 20 & 40 & 80 & 160 
    & 10 & 20 & 40 & 80 & 160 \\ \cline{2-13}
    & \multicolumn{2}{c|}{$\theta$} & \multicolumn{5}{c||}{-1: symmetric}  & \multicolumn{5}{c}{1: nonsymmetric} \\ \hline
    \multirow{5}{*}{1} & \multirow{3}{*}{$A_0^{-1}$} & GMG & 
    36 & 46 & 52 & 57 & 59 &
    40 & 48 & 57 & 62 & 66 \\ \cline{3-13}
    & & BoomerAMG & 
    37 & 46 & 54 & 58 & 60 &
    41 & 48 & 56 & 60 & 64 \\ \cline{3-13}
    & & DirectSolver & 
    36 & 44 & 52 & 56 & 60 & 
    39 & 47 & 54 & 59 & 62 \\ \cline{2-13}        
    & \multicolumn{2}{c|}{BoomerAMG} & 
    15 & 24 & 34 & 46 & 66 &
    16 & 19 & 28 & 43 & 69 \\ \cline{2-13}        
    & \multicolumn{2}{c|}{Identity} & 
    238 & 537 & 1036 & 2049 & - & 
    280 & 570 & 1098 & 2548 & -\\ \hline
    \multirow{5}{*}{2} & \multirow{3}{*}{$A_0^{-1}$} & GMG & 
    62 & 79 & 88 & 92 & 97 &
    68 & 81 & 89 & 93 & 96 \\
    \cline{3-13}
    & & BoomerAMG &
    61 & 78 & 89 & 95 & 100 &
    66 & 84 & 92 & 97 & 98 \\ \cline{3-13}
    & & DirectSolver & 
    57 & 76 & 86 & 93 & 96 &
    63 & 82 & 89 & 93 & 94 \\ \cline{2-13}        
    & \multicolumn{2}{c|}{BoomerAMG} & 
    39 & 55 & 67 & 92 & 126 & 
    36 & 48 & 55 & 67 & 90 \\ \cline{2-13}        
    & \multicolumn{2}{c|}{Identity} &
    338 & 691 & 1412 & 2758 & - & 
    297 & 711 & 1613 & - & - \\ \hline
    \multirow{5}{*}{3} & \multirow{3}{*}{$A_0^{-1}$} & GMG & 
    98 & 115 & 122 & 129 & 131 &
    99 & 118 & 125 & 129 & 129 \\
    \cline{3-13}
    & & BoomerAMG & 
    98 & 116 & 127 & 133 & 136 &
    106 & 121 & 131 & 133 & 133 \\ 
    \cline{3-13}        
    & & DirectSolver & 
    94 & 113 & 123 & 129 & 131 & 
    102  & 117 & 126 & 129 & 129 \\ \cline{2-13}        
    & \multicolumn{2}{c|}{BoomerAMG} & 
    44 & 65 & 78 & 99 & 130 & 
    46 & 62 & 70 & 82 & 105 \\ \cline{2-13}        
    & \multicolumn{2}{c|}{Identity} & 
    437 & 905 & 1786 & - & - &
    390 & 942 & 2229 & - & - \\ \hline
    \multirow{5}{*}{4} & \multirow{3}{*}{$A_0^{-1}$} & GMG &
    142 & 178 & 187 & 199 & 204 &
    129 & 166 & 184 & 200 & 206 \\ \cline{3-13}
    & & BoomerAMG & 
    122 & 161 & 183 & 197 & 203 & 
    131 & 168 & 191 & 205 & 210 \\ \cline{3-13}
    & & DirectSolver &
    122 & 157 & 176 & 192 & 198 & 
    129 & 163 & 186 & 200 & 208 \\ \cline{2-13}        
    & \multicolumn{2}{c|}{BoomerAMG} & 
    99 & 171 & 211 & 239 & 273 & 
    72 & 132 & 161 & 181 & 209 \\ \cline{2-13}        
    & \multicolumn{2}{c|}{Identity} &
    577 & 1301 & 2584 & - & - & 
    469 & 1062 & - & - & - \\
    \hline\hline
  \end{tabular}
  }}
  \caption{The convergence steps for symmetric and nonsymmetric methods in Example 6.}
  \label{tab_ex6_steps}
\end{table}

\begin{table}[htbp]
  \centering
  \renewcommand{\arraystretch}{1.15}
  \scalebox{0.75}{
  \revise{}{
  \begin{tabular}{p{0.2cm}|p{0.5cm}|p{2.5cm}|p{1.05cm}|p{1.05cm}|p{1.05cm}|p{1.05cm} | p{1.05cm}||
    p{1.05cm}|p{1.05cm}|p{1.05cm}|p{1.05cm}|p{1.05cm}}
    \hline\hline
    \multirow{2}{*}{$m$} & \multicolumn{2}{c|}{\diagbox[width=3.5cm]{\small Preconditioner}{$1/h$}} & 10 & 20 & 40 & 80 & 
    160 & 10 & 20 & 40 & 80 & 160 \\ \cline{2-13}
    & \multicolumn{2}{c|}{$\theta$} & \multicolumn{5}{c||}{-1:symmetric}  & \multicolumn{5}{c}{1:nonsymmetric} \\ \hline
    \multirow{5}{*}{1} & \multirow{3}{*}{$A_0^{-1}$} & GMG & 
    0.006 & 0.028 & 0.128 & 0.689 & 2.736 &
    0.028 & 0.046 & 0.148 & 0.653 & 2.969 \\ \cline{3-13}
    & & BoomerAMG & 
    0.010 & 0.046 & 0.200 & 0.913 & 3.812 & 
    0.011 & 0.042 & 0.213 & 0.925 & 3.992 \\ \cline{3-13}
    & & DirectSolver & 
    0.004 & 0.025 & 0.130 & 0.926 & 6.034 & 
    0.005 & 0.023 & 0.118 & 1.007 & 5.974 \\\cline{2-13}        
    & \multicolumn{2}{c|}{BoomerAMG} & 
    0.011 & 0.060 & 0.284 & 1.777 & 11.39 & 
    0.013 & 0.052 & 0.271 & 1.837 & 12.81 \\ \cline{2-13}        
    & \multicolumn{2}{c|}{Identity} & 
    0.020 & 0.183 & 1.408 & 15.19 & - &
    0.038 & 0.322 & 2.665 & 35.41 & - \\\hline
    \multirow{5}{*}{2} & \multirow{3}{*}{$A_0^{-1}$} & GMG & 
    0.011 & 0.053 & 0.242 & 1.172 & 4.993 &
    0.044 & 0.084 & 0.266 & 1.173 & 5.093 \\ \cline{3-13}
    & & BoomerAMG & 
    0.020 & 0.078 & 0.360 & 1.597 & 7.221 &
    0.018 & 0.084 & 0.387 & 1.639 & 7.562 \\ \cline{3-13}        
    & & DirectSolver & 
    0.009 & 0.047 & 0.245 & 1.692 & 10.20 & 
    0.009 & 0.047 & 0.236 & 1.943 & 12.32 \\\cline{2-13}        
    & \multicolumn{2}{c|}{BoomerAMG} &
    0.027 & 0.152 & 0.721 & 4.698 & 22.56 & 
    0.030 & 0.155 & 0.706 & 4.154 & 20.40 \\ \cline{2-13}        
    & \multicolumn{2}{c|}{Identity} &
    0.035 & 0.281 & 2.378 & 18.93 & - & 
    0.066 & 0.417 & 4.343 & - & - \\\hline
    \multirow{5}{*}{3} & \multirow{3}{*}{$A_0^{-1}$} & GMG &
    0.024 & 0.111 & 0.486 & 2.303 & 10.59 & 
    0.059 & 0.143 & 0.488 & 2.403 & 10.64 \\ \cline{3-13}
    & & BoomerAMG & 
    0.036 & 0.147 & 0.639 & 3.004 & 13.11 &
    0.031 & 0.117 & 0.690 & 2.912 & 13.71 \\ \cline{3-13}        
    & & DirectSolver & 
    0.020 & 0.103 & 0.514 & 3.826 & 15.56 & 
    0.018 & 0.083 & 0.436 & 3.769 & 15.20 \\\cline{2-13}        
    & \multicolumn{2}{c|}{BoomerAMG} & 
    0.041 & 0.217 & 1.074 & 7.294 & 41.81 & 
    0.042 & 0.199 & 0.960 & 5.696 & 36.81 \\ \cline{2-13}        
    & \multicolumn{2}{c|}{Identity} & 
    0.072 & 0.618 & 5.140 & - & - &
    0.065 & 0.705 & 7.763 & - & - \\ \hline
    \multirow{5}{*}{4} & \multirow{3}{*}{$A_0^{-1}$} & GMG &
    0.041 & 0.221 & 1.010 & 5.302 & 21.28 & 
    0.081 & 0.237 & 0.909 & 4.934 & 21.35  \\ \cline{3-13}
    & & BoomerAMG & 
    0.055 & 0.277 & 1.413 & 6.208 & 24.41 & 
    0.044 & 0.215 & 1.085 & 5.520 & 31.75 \\ \cline{3-13}
    & & DirectSolver & 
    0.033 & 0.185 & 1.012 & 6.038 & 27.45 & 
    0.026 & 0.143 & 0.866 & 5.854 & 27.10 \\\cline{2-13}        
    & \multicolumn{2}{c|}{BoomerAMG} & 
    0.109 & 0.788 & 4.583 & 25.08 & 115.20 & 
    0.089 & 0.703 & 3.865 & 20.66 & 89.13 \\ \cline{2-13}        
    & \multicolumn{2}{c|}{Identity} & 
    0.133 & 1.264 & 10.66 & - & - &
    0.092 & 0.956 & - & - & - \\
    \hline\hline
  \end{tabular}}}
  \caption{The CPU times of solving linear system for symmetric and
  nonsymmetric methods in Example 6.}
  \label{tab_ex6_times}
\end{table}

\noindent \textbf{Example 7.}
We solve a three-dimensional elliptic problem defined in the cubic
domain $\Omega = (0, 1)^3$.  \substitute{ with the analytic solution
\begin{displaymath}
  u(x, y, z) = \sin(x + y + z),
\end{displaymath}
}{The exact solution is selected to be the same as Example 2}.  
This problem is solved on a series of tetrahedral meshes
with the mesh size $h = 1/4, 1/8, 1/16, 1/32$. 
The numerical results are reported in Tab.~\ref{tab_ex7_steps}.
In three dimensions, the proposed
preconditioning method still has a fast convergence speed, and the
convergence steps for symmetric/nonsymmetric systems keep almost
unchanged as $h$ tends to zero.  This numerical observation validates
the estimates in Theorem \ref{th_A0Amcond}. 
The convergence histories for all accuracy $1 \leq m \leq 3$ are
plotted in Fig.~\ref{fig_ex7_relative_error}. The relative errors in
iterations decrease sharply for our method. 
The CPU times costed in solving
linear systems are displayed in Tab.~\ref{tab_ex7_times}.  Clearly,
our method is still efficient for problems in three dimensions.

\begin{table}[htbp]
  \centering
  \renewcommand{\arraystretch}{1.15}
  \scalebox{0.75}{
  \revise{}{
  \begin{tabular}{p{0.6cm}|p{0.5cm}|p{2.5cm}|p{0.8cm}|p{0.8cm}|p{0.8cm}|p{0.8cm}||p{0.8cm}|p{0.8cm}|p{0.8cm}|p{0.8cm}}
    \hline\hline
    \multirow{2}{*}{$m$} & \multicolumn{2}{c|}{\diagbox[width=3.85cm]{Preconditioner}{$1/h$}} &
    4 & 8 & 16 & 32 & 
    4 & 8 & 16 & 32 \\ \cline{2-11}
    & \multicolumn{2}{c|}{$\theta$} & \multicolumn{4}{c||}{-1:symmetric}  & \multicolumn{4}{c}{1:nonsymmetric} \\ \hline
    \multirow{5}{*}{1} & \multirow{3}{*}{$A_0^{-1}$} & GMG  &
    18 & 26 & 32 & 35 &
    23 & 27 & 29 & 30 \\ \cline{3-11}
    & & BoomerAMG  &
    18 & 25 & 32 & 36 &
    23 & 26 & 28 & 28 \\ \cline{3-11}  
    & & DirectSolver & 
    17 & 24 & 31 & 34 &
    23 & 26 & 27 & 28 \\ \cline{2-11}  
    & \multicolumn{2}{c|}{BoomerAMG} & 
    9 & 15 & 26 & 36 & 
    8 & 13 & 23 & 33  \\ \cline{2-11}        
    & \multicolumn{2}{c|}{Identity} & 
    155 & 399 & 788 & 1576 & 
    75  & 183 & 423 & 1333 \\ \hline
    \multirow{5}{*}{2} & \multirow{3}{*}{$A_0^{-1}$} & GMG &
    36 & 46 & 49 & 52 &
    40 & 44 & 46 & 47\\ \cline{3-11}
    & & BoomerAMG  & 
    36 & 46 & 48 & 51 & 
    40 & 43 & 43 & 44 \\ \cline{3-11}        
    & & DirectSolver &
    36 & 45 & 46 & 50 &
    40 & 43 & 44 & 44 \\ \cline{2-11}        
    & \multicolumn{2}{c|}{BoomerAMG} &
    18 & 28 & 42 & 66 & 
    18 & 26 & 40 & 62 \\ \cline{2-11}        
    & \multicolumn{2}{c|}{Identity} & 
    186 & 522 & 1044 & 2047 & 
    160 & 493 & 1093 & 2460 \\ \hline
    \multirow{5}{*}{3} & \multirow{3}{*}{$A_0^{-1}$} & GMG  & 
    61 & 89 & 93 & 96 & 
    48 & 57 & 64 & 67 \\ \cline{3-11}
    & & BoomerAMG  & 
    61 & 81 & 90 & 95 &
    47 & 55 & 61 & 62 \\ \cline{3-11}        
    & & DirectSolver  & 
    61 & 80 & 88 & 93 &
    48 & 56 & 61 & 63 \\ \cline{2-11}        
    & \multicolumn{2}{c|}{BoomerAMG} & 
    43 & 64 & 75 & 92 & 
    46 & 69 & 76 & 88  \\ \cline{2-11}        
    & \multicolumn{2}{c|}{Identity} & 
    277 & 812 & 1564 & - &
    201 & 684 & 1417 & - \\ 
    \hline\hline
  \end{tabular}}}
  \caption{The convergence steps for symmetric and nonsymmetric methods in Example 7.}
  \label{tab_ex7_steps}
\end{table}

\begin{table}[htbp]
  \centering
  \renewcommand{\arraystretch}{1.15}
  \scalebox{0.8}{
  \revise{}{
  \begin{tabular}{p{0.4cm}|p{0.5cm}|p{2.5cm}|p{0.9cm}|p{0.9cm}|p{0.9cm}|p{0.9cm}||p{0.9cm}|p{0.9cm}|p{0.9cm}|p{0.9cm}}
    \hline\hline
    \multirow{2}{*}{$m$} & \multicolumn{2}{c|}{\diagbox[width=3.5cm]{\small Preconditioner}{$1/h$}} & 4 & 8 & 16 & 32 & 
    4 & 8 & 16 & 32 \\ \cline{2-11}
    & \multicolumn{2}{c|}{$\theta$} & \multicolumn{4}{c||}{-1:symmetric}  & \multicolumn{4}{c}{1:nonsymmetric} \\ \hline
    \multirow{5}{*}{1} & \multirow{3}{*}{$A_0^{-1}$} & GMG  & 
    0.003 & 0.022 & 0.218 & 2.092 & 
    0.003 & 0.018 & 0.202 & 1.969 \\ \cline{3-11}
    & & BoomerAMG  & 
    0.003 & 0.038 & 0.500 & 6.832 & 
    0.004 & 0.026 & 0.306 & 4.045 \\ \cline{3-11}        
    & & DirectSolver  & 
    0.002 & 0.029 & 0.832 & 17.29 & 
    0.002 & 0.020 & 0.627 & 12.94\\ \cline{2-11}  
    & \multicolumn{2}{c|}{BoomerAMG} & 
    0.003 & 0.051 & 1.182 & 19.53 & 
    0.004 & 0.056 & 1.200 & 16.59  \\ \cline{2-11}        
    & \multicolumn{2}{c|}{Identity} & 
    0.008 & 0.192 & 4.691 & 77.80 & 
    0.004 & 0.107 & 2.999 & 102.25 \\ \hline
    \multirow{5}{*}{2} & \multirow{3}{*}{$A_0^{-1}$} & GMG & 
    0.007 & 0.067 & 0.659 & 5.684 & 
    0.006 & 0.038 & 0.424 & 4.029 \\
    \cline{3-11}
    & & BoomerAMG  & 
    0.010 & 0.089 & 0.969 & 8.216 &
    0.009 & 0.069 & 0.825 & 7.336 \\ \cline{3-11}        
    & & DirectSolver  &
    0.005 & 0.070 & 1.553 & 25.31 &
    0.004 & 0.040 & 1.001 & 22.94 \\ \cline{2-11}        
    & \multicolumn{2}{c|}{BoomerAMG} & 
    0.010 & 0.165 & 3.288 & 53.99 & 
    0.009 & 0.175 & 3.207 & 47.28  \\ \cline{2-11}        
    & \multicolumn{2}{c|}{Identity} & 
    0.018 & 0.529 & 10.88 & 141.10 &
    0.013 & 0.379 & 9.446 & 238.85 \\ \hline
    \multirow{5}{*}{3} & \multirow{3}{*}{$A_0^{-1}$} & GMG  &
    0.022 & 0.220 & 2.837 & 23.39 &
    0.019 & 0.165 & 1.705 & 16.52 \\ \cline{3-11}
    & & BoomerAMG  & 
    0.018 & 0.241 & 2.742 & 30.52 &
    0.015 & 0.185 & 1.826 & 21.15 \\ \cline{3-11}
    & & DirectSolver  &
    0.014 & 0.198 & 3.603 & 62.84 & 
    0.009 & 0.135 & 1.985 & 36.04  \\ \cline{2-11}        
    & \multicolumn{2}{c|}{BoomerAMG} & 
    0.035 & 0.750 & 12.54 & 139.98 & 
    0.039 & 0.861 & 12.22 & 121.88 \\ \cline{2-11}        
    & \multicolumn{2}{c|}{Identity} &
    0.050 & 1.709 & 42.10 & - & 
    0.027 & 1.057 & 21.84 & -  \\ 
    \hline\hline
  \end{tabular}}}
  \caption{The CPU times for symmetric and nonsymmetric methods in Example 7.}
  \label{tab_ex7_times}
\end{table}

\begin{figure}[htbp]
  \centering
  \includegraphics[width=0.35\textwidth]{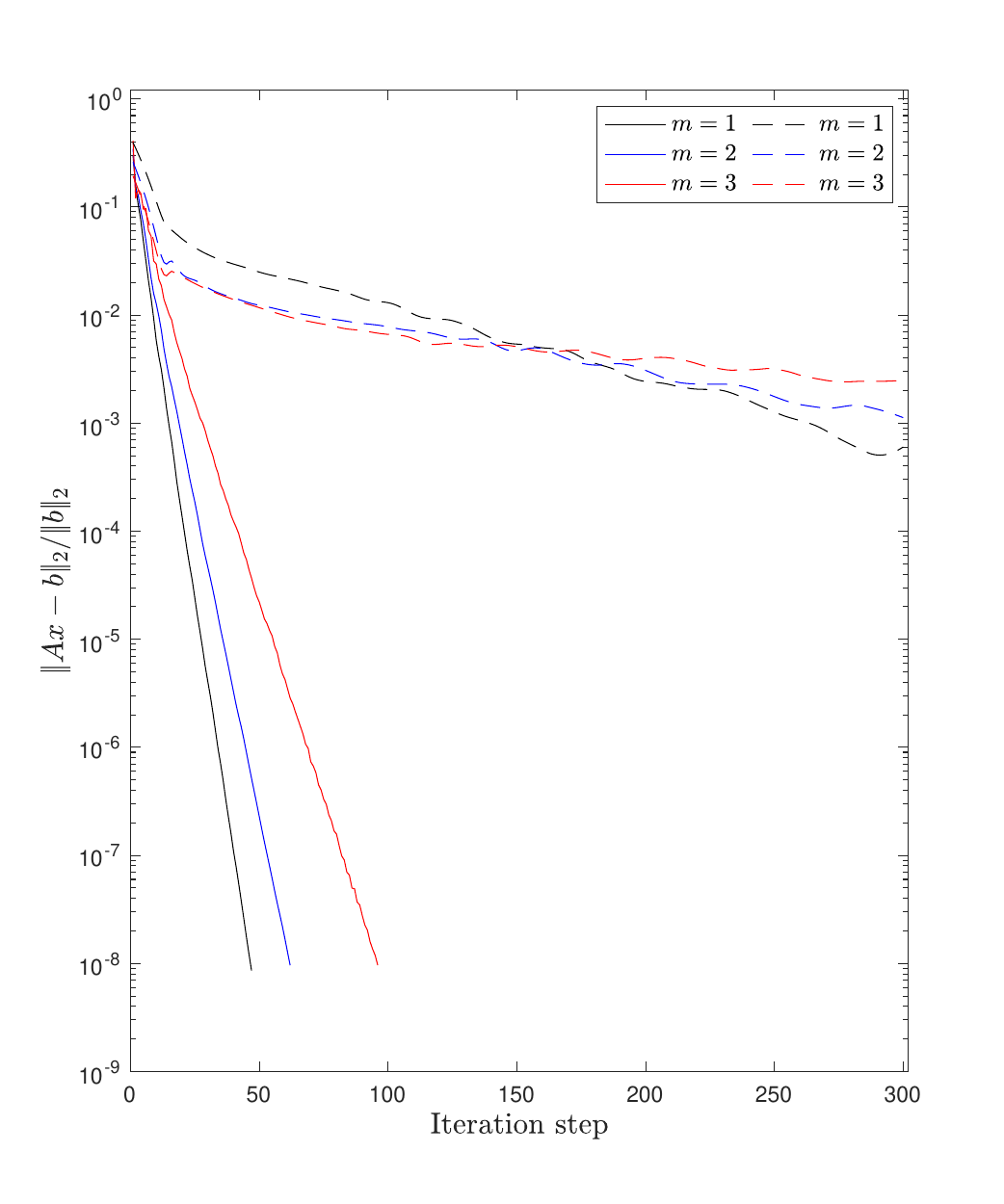}
  \hspace{20pt}
  \includegraphics[width=0.35\textwidth]{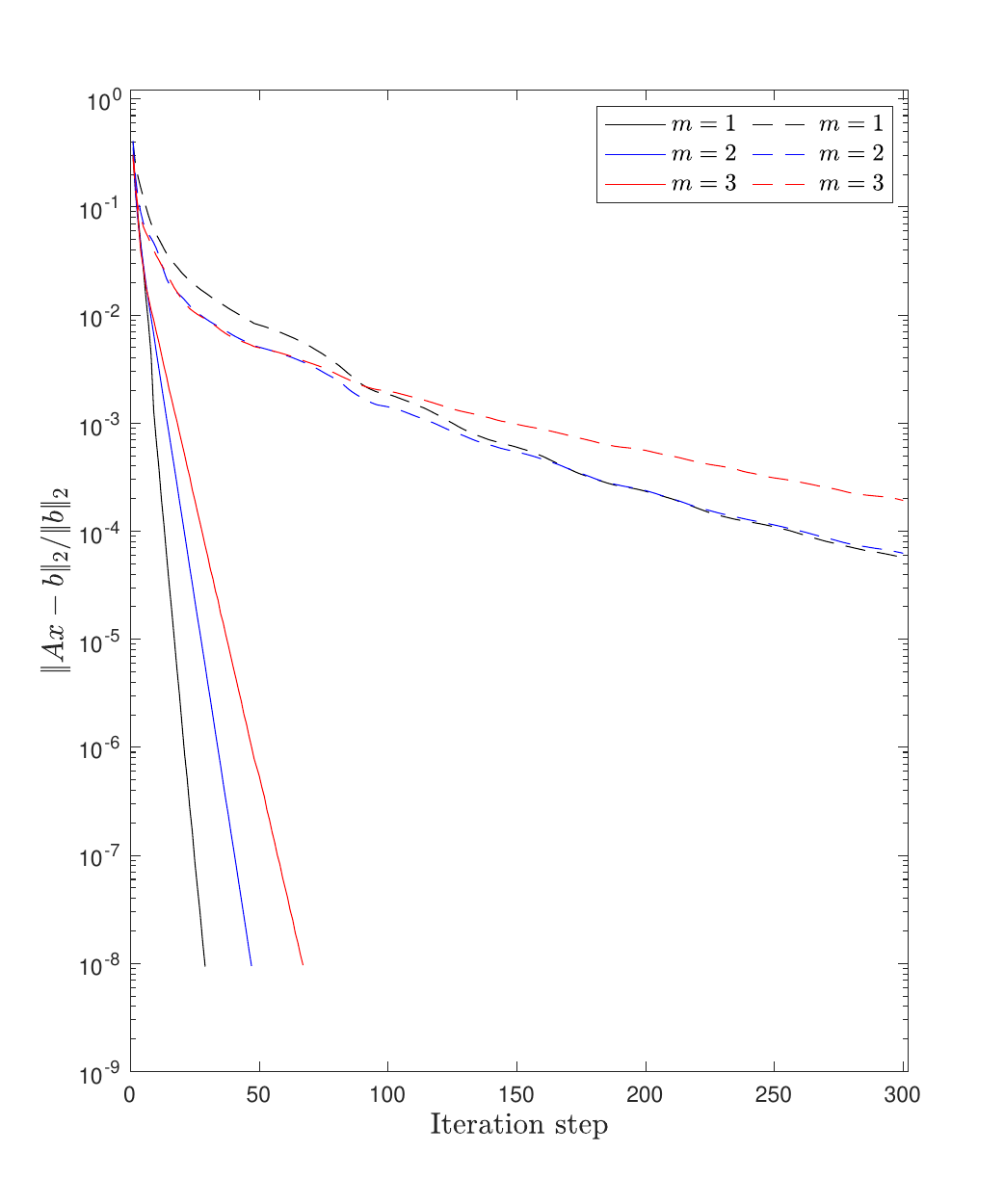}
  \caption{The convergence histories of CG (left) /GMRES (right) solvers on the mesh $h =
  1/32$ (solid line: CG/GMRES solver with preconditioner $A_0^{-1}$,
  dashed line: standard CG/GMRES solver).}
  \label{fig_ex7_relative_error}
\end{figure}


\section{Conclusions}
\label{sec_conclusion}
In this paper, we proposed a preconditioned interior penalty method
for the elliptic problem with the reconstructed discontinuous
approximation.  The schemes are derived under the
symmetric/nonsymmetric interior penalty DG methods. We constructed a
preconditioner from the piecewise constant function and the
preconditioned system is shown to be optimal. Numerical experiments
demonstrated the efficiency on both the approximation by the
reconstructed space and the iterative method for the preconditioned
system. 

\section*{Acknowledgements} 
The authors would like to thank the anonymous referees
sincerely for their constructive comments that improve the quality of
this paper. 
This research was supported by National Natural Science Foundation of
China (12201442, 12288101).

\section*{Data Availability} 
The author declares that all data
supporting the findings of this study are available within this
particle.

\section*{Declarations}
The author has no relevant financial or non-financial interests to
disclose.


\bibliographystyle{amsplain}
\bibliography{../ref}

\end{document}